\documentclass[reqno,10pt,a4paper, english]{amsart}

\usepackage{amsmath} % utilities for mathematics
\usepackage{amsthm} % utilities for theorem environment
\usepackage{amssymb} % loads mathematical fonts
\usepackage[dvipsnames]{xcolor}
\usepackage{amscd} % utilities for commutative diagrams
\usepackage{mathtools}
\usepackage[ddmmyyyy,hhmmss]{datetime}
\usepackage[colorinlistoftodos,bordercolor=orange,backgroundcolor=orange!20,linecolor=orange,textsize=scriptsize]{todonotes}

\usepackage{subcaption}% <-- added

\usepackage{array} % core package

\usepackage{stix2}

\usepackage[cal=boondoxo,scr=euler]{mathalfa}
\usepackage[linktocpage]{hyperref} % permits navigating on the pdf
\usepackage{cleveref} % smart referencing
\usepackage{caption} %
\usepackage{graphics,graphicx} % permits putting images
\usepackage{tikz,tikz-cd, tikz-3dplot} % tool for diagrams
\usepackage{contour} % contour around text for diagrams
\contourlength{0.125em}
\usetikzlibrary{arrows,arrows.meta}

\usepackage{float} % better management of the position of a figure

\usetikzlibrary{fit,matrix,graphs,graphs.standard,calc,shapes.geometric, arrows}
\usetikzlibrary{decorations.pathreplacing}%angles,quotes

\usetikzlibrary{automata}
\usepackage[shortlabels]{enumitem} 

\DeclareMathAlphabet{\mathsf}{OT1}{\sfdefault}{m}{n}

\newcommand{\nocontentsline}[3]{}
\newcommand{\tocless}[2]{\bgroup\let\addcontentsline=\nocontentsline#1{#2}\egroup}

\usepackage[margin=1.4in]{geometry}

\usepackage{verbatim}

\usepackage{scalerel}

\usepackage{multirow}

\makeatletter
\def\dual#1{\expandafter\dual@aux#1\@nil}
\def\dual@aux#1/#2\@nil{\begin{tabular}{@{}c@{}}#1\\#2\end{tabular}}
\makeatother

\makeatletter
\@namedef{subjclassname@2020}{\textup{2020} Mathematics Subject Classification}
\makeatother

\renewcommand{\tilde}{\widetilde}

\DeclareMathAlphabet{\amathbb}{U}{bbold}{m}{n}

\hypersetup{
    colorlinks = true,
    linkbordercolor = {white},
    linkcolor = {BrickRed},
    anchorcolor = {black},
    citecolor = {BrickRed},
    filecolor = {cyan},
    menucolor = {BrickRed},
    runcolor = {cyan},
    urlcolor = {black}
}

\tikzstyle{rectan} = [rectangle, rounded corners, 
minimum width=1.5cm, 
minimum height=0.75cm,
text width=3cm,
text centered, 
draw=black,
font = \footnotesize,
]

\tikzstyle{ghost} = [circle, 
minimum width=1pt, 
minimum height=1pt,
text width=1pt,
text centered, 
draw=black,
font = \footnotesize,
]

\newtheoremstyle{teoremas}% <name>
{11pt}% <Space above>
{11pt}% <Space below>
{\itshape}% <Body font>
{}% <Indent amount>
{\bfseries}% <Theorem head font>
{}% <Punctuation after theorem head>
{.5em}% <Space after theorem headi>
{}% <Theorem head spec (can be left empty, meaning `normal')>

\theoremstyle{teoremas}
\newtheorem{theorem}{Theorem}[section]
\newtheorem{corollary}[theorem]{Corollary}
\newtheorem{lemma}[theorem]{Lemma}
\newtheorem{proposition}[theorem]{Proposition}

\newtheoremstyle{definition}% <name>
{11pt}% <Space above>
{11pt}% <Space below>
{}% <Body font>
{}% <Indent amount>
{\bfseries}% <Theorem head font>
{}% <Punctuation after theorem head>
{.5em}% <Space after theorem headi>
{}% <Theorem head spec (can be left empty, meaning `normal')>

\theoremstyle{definition}
\newtheorem{definition}[theorem]{Definition}

\newtheorem{example}[theorem]{Example}
\newtheorem{remark}[theorem]{Remark}

\crefname{theorem}{theorem}{theorems}
\Crefname{theorem}{Theorem}{Theorems}
\crefname{lemma}{lemma}{lemmas}
\Crefname{lemma}{Lemma}{Lemmas}
\crefname{proposition}{proposition}{propositions}
\Crefname{proposition}{Proposition}{Propositions}

\DeclareMathOperator{\Bier}{Bier}

\newcommand{\R}{\mathbb{R}}
\newcommand{\Z}{\mathbb{Z}}

\renewcommand{\H}{\mathrm{H}}

\newcommand{\Int}{\operatorname{Int}}

\newcommand{\rev}{\operatorname{rev}}

\newcommand{\cF}{\mathcal{F}}

\newcommand{\zero}{\widehat{0}}
\newcommand{\one}{\widehat{1}}

\let\oldhat\hat
\renewcommand{\hat}[1]{\widehat{#1}}

\AtBeginDocument{%
   \def\MR#1{}
}

\newcommand{\Stell}{\operatorname{Stell}}
\newcommand{\DStell}{\operatorname{DStell}}

\title{Stellahedral geometry of partially ordered sets}

\author{Tommaso Faustini}
\address{Department of Mathematics, University of Warwick, Coventry, United Kingdom}
\email{tommaso.faustini@warwick.ac.uk}

\author{Luis Ferroni}
\address{Dipartimento di Matematica, Universit\`a di Pisa, Pisa, Italy}
\email{luis.ferroni@unipi.it}

\author{Ludovico Piazza}
\address{Dipartimento di Matematica, Universit\`a di Pisa, Pisa, Italy}
\email{l.piazza5@studenti.unipi.it}
\thanks{}

\subjclass[2020]{Primary: 06A07, 05B35, 52B05, 05A20}

\allowdisplaybreaks
\begin{document}

\begin{abstract}
    We introduce a transformation on partially ordered sets, termed the \emph{stellahedral transform}, with notable features. 
    It preserves the properties of being Eulerian, Cohen--Macaulay, and of being the face poset of a polytope. 
    Furthermore, it admits an explicit geometric realization for convex polytopes and specializes to the construction that takes a simplex to the stellahedron. One motivation for this definition comes from the theory of toric $h$-polynomials and (augmented) Chow polynomials of Eulerian posets. 
    We show that the right augmented Chow polynomial of an Eulerian poset $P$ agrees with the toric $h$-polynomial of the stellahedral transform of $P$. We use $\mathbf{cd}$-index results due to Ehrenborg (2005) and Karu (2006), to prove two positivity results for augmented Chow polynomials: for Gorenstein* posets they are unimodal, and for face posets of polytopes they are $\gamma$-positive. 
    
    Along the way we provide negative answers to two open questions concerning Eulerian and Gorenstein* posets. First, the question on the nonnegativity of Eulerian Chow polynomials, posed by Ferroni, Matherne, and Vecchi (2024). Second, the questions posed by Brenti and Welker (2008) and Athanasiadis and Kalampogia-Evangelinou (2023) on the real-rootedness of chain and Chow polynomials of face posets of polytopes and Gorenstein* posets.
\end{abstract}

\date{}

\maketitle

\section{Introduction}\label{sec:one}

\subsection*{Overview}

Over the last decade, the geometry of the permutahedron and the stellahedron has played a prominent role in several developments in matroid theory, particularly in its Hodge-theoretic aspects (see \cite{adiprasito-huh-katz,stellahedral,braden-huh-matherne-proudfoot-wang}). For the Boolean matroid, the cohomology rings of the associated toric varieties recover, respectively, its Chow ring and its augmented Chow ring. More generally, Chow rings and augmented Chow rings are central objects in the Hodge theory of matroids and enter, in particular, into the construction of matroid intersection cohomology.

The purpose of this paper is to introduce a poset-theoretic construction that extends the stellahedral side of this picture beyond the Boolean case. We introduce a transformation, which we call the \emph{stellahedral transform}, that associates with a bounded graded poset $P$ of rank $r$ a bounded graded poset $\Stell(P)$ of rank $r+1$.

\begin{definition}
Let $P$ be a poset with minimum element ${\zero}$. We define the \emph{stellahedral transform} of $P$ as the poset
\[
\Stell(P)= \left\{ (c,y):  c=\left\{{\zero} =c_0 < c_1 < \dots < c_k\right\} \text{ chain in $P$},\,y \in P,\, c_k \leq y\right\} \sqcup \left\{{\one}_{\Stell}\right\},
\]
where ${\one}_{\Stell}$ is a maximum element and
\[
(c,y) \le (c',y') \iff c \subseteq c' \text{ and } y \ge y' \text{ in $P$}.
\]
\end{definition}

The definition combines the order relation of $P$ with the combinatorics of chains in $P$, and is inspired by (but is not directly related to) the definition of the augmented Bergman complex of a matroid, by Braden, Huh, Matherne, Proudfoot, and Wang \cite{semismall}. At first glance there is little reason to expect the resulting poset to actually retain any good structural properties that $P$ possesses. As it turns out, the transformation in fact does preserve many of them.

\begin{theorem}\label{thm:intro-preservation}
The stellahedral transform preserves the properties of being Eulerian and Cohen--Macaulay. In particular, it preserves the Gorenstein* property.
\end{theorem}

The polyhedral geometric behavior of the construction is even more concrete. For a convex polytope $\mathcal{P}$, we denote its face poset by $\cF(\mathcal{P})$ and its polar dual by $\mathcal{P}^*$, taking an interior point as the origin in its affine span. The following result says that if $P$ is the face poset of a polytope, then so is $\Stell(P)$. Moreover, the corresponding polytope is obtained by a canonical sequence of stellar subdivisions of the pyramid over its polar dual.

\begin{theorem}\label{thm:intro-polytopal}
Let $\mathcal{P}\subseteq \mathbb{R}^n$ be a full-dimensional convex
polytope with face poset $P=\cF(\mathcal{P})$. Then
\[
\Stell(P) \cong \cF\bigl(\DStell(\mathcal{P}^*)\bigr).
\]
Here $\DStell(\mathcal{P}^*)$ is obtained by taking the pyramid
$\operatorname{Pyr}(\mathcal{P}^*)$ over the polar dual of $\mathcal{P}$ and then performing stellar subdivisions
at all proper faces of the pyramid that strictly contain its apex, in decreasing order of dimension.
\end{theorem}

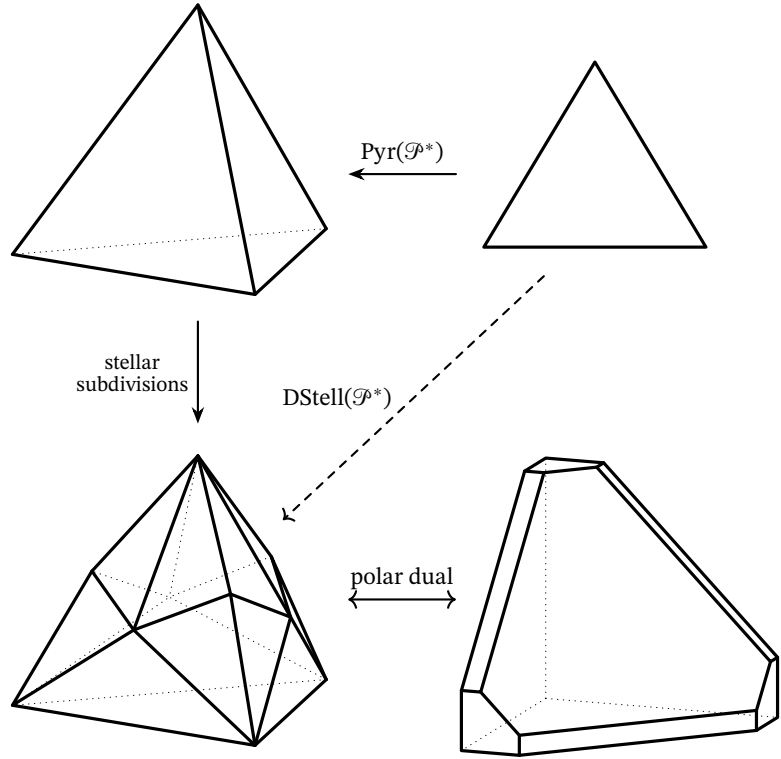
\begin{figure}[ht]
    \centering

    \begin{tikzpicture}[scale = 0.75,
        line join=round,
        line cap=round,
        diagram arrow/.style={
            -{Stealth[length=2.2mm]},
            thick
        },
        object label/.style={font=\small}
    ]

    \tdplotsetmaincoords{77}{160}

    \begin{scope}[
        shift={(3.2cm,5.5cm)},
        tdplot_main_coords,
        scale=0.92
    ]
            \def\xyfactor{0.8}

        % Base vertices
        \coordinate (DA) at
            ({-4*\xyfactor},{-2*\xyfactor},0);
        \coordinate (DB) at
            ({ 4*\xyfactor},{-2*\xyfactor},0);
        \coordinate (DC) at
            (0,{4*\xyfactor},0);

        % Intermediate vertices
        \coordinate (DD) at
            ({-16*\xyfactor/7},{-8*\xyfactor/7},{75/28});
        \coordinate (DE) at
            ({-2*\xyfactor},{\xyfactor},{15/8});
        \coordinate (DF) at
            (0,{-2*\xyfactor},{15/8});
        \coordinate (DH) at
            (0,{16*\xyfactor/7},{75/28});
        \coordinate (DI) at
            ({2*\xyfactor},{\xyfactor},{15/8});
        \coordinate (DJ) at
            ({16*\xyfactor/7},{-8*\xyfactor/7},{75/28});

        % Apex
        \coordinate (DG) at (0,0,5);

        % Hidden edges
        \draw[dotted,very thin] (DA)--(DB);

        % Visible edges
        \draw[very thick] (DB)--(DC)--(DA);

        \draw[very thick] (DA)--(DG)--(DB);
        \draw[very thick] (DC)--(DG);
    \end{scope}

    % ========================================================
    % (1) The triangle Delta_2
    % ========================================================

    \begin{scope}[shift={(10.2cm,6.80cm)}, scale=1.45]
        \coordinate (A) at (-1.35,-0.75);
        \coordinate (B) at ( 1.35,-0.75);
        \coordinate (C) at ( 0, 1.50);

        \draw[very thick] (A)--(B)--(C)--cycle;
    \end{scope}

    % ========================================================
    % (3) DStell(Delta_2)
    % ========================================================

    \tdplotsetmaincoords{77}{160}

    \begin{scope}[
        shift={(3.2cm,-2.45cm)},
        tdplot_main_coords,
        scale=0.92
    ]
        \def\xyfactor{0.8}

        % Base vertices
        \coordinate (DA) at
            ({-4*\xyfactor},{-2*\xyfactor},0);
        \coordinate (DB) at
            ({ 4*\xyfactor},{-2*\xyfactor},0);
        \coordinate (DC) at
            (0,{4*\xyfactor},0);

        % Intermediate vertices
        \coordinate (DD) at
            ({-16*\xyfactor/7},{-8*\xyfactor/7},{75/28});
        \coordinate (DE) at
            ({-2*\xyfactor},{\xyfactor},{15/8});
        \coordinate (DF) at
            (0,{-2*\xyfactor},{15/8});
        \coordinate (DH) at
            (0,{16*\xyfactor/7},{75/28});
        \coordinate (DI) at
            ({2*\xyfactor},{\xyfactor},{15/8});
        \coordinate (DJ) at
            ({16*\xyfactor/7},{-8*\xyfactor/7},{75/28});

        % Apex
        \coordinate (DG) at (0,0,5);

        % Hidden edges
        \draw[dotted,very thin] (DA)--(DB);
        \draw[dotted,very thin] (DD)--(DF);
        \draw[dotted,very thin] (DB)--(DF)--(DJ);
        \draw[dotted,very thin] (DA)--(DF)--(DG);

        % Visible edges
        \draw[very thick] (DB)--(DC)--(DI)--cycle;
        \draw[very thick] (DB)--(DJ);
        \draw[very thick] (DC)--(DH)--(DI);
        \draw[very thick] (DG)--(DJ);
        \draw[very thick] (DI)--(DJ);
        \draw[very thick] (DH)--(DG)--(DI);
        \draw[very thick] (DC)--(DE)--(DG);
        \draw[very thick] (DE)--(DH);
        \draw[very thick] (DA)--(DC);
        \draw[very thick] (DA)--(DD)--(DE);
        \draw[very thick] (DA)--(DE);
        \draw[very thick] (DD)--(DG);
    \end{scope}

    % ========================================================
    % (4) The three-dimensional stellahedron
    % ========================================================

    \tdplotsetmaincoords{77}{110}

    \begin{scope}[
        shift={(10.2cm,-1.25cm)},
        tdplot_main_coords,
        scale=5.8
    ]
        \coordinate (SA) at ({-1/4},{-1/4},{-1/4});
        \coordinate (SB) at ({-1/4},{-1/4},{ 1/2});
        \coordinate (SC) at ({-1/4},{-1/16},{1/2});
        \coordinate (SD) at ({-1/4},{ 1/2},{-1/4});
        \coordinate (SE) at ({-1/4},{ 1/2},{-1/16});
        \coordinate (SF) at ({-3/16},{-1/16},{1/2});
        \coordinate (SG) at ({-3/16},{ 1/2},{-1/16});
        \coordinate (SH) at ({-1/16},{-1/4},{1/2});
        \coordinate (SI) at ({-1/16},{-3/16},{1/2});
        \coordinate (SJ) at ({-1/16},{ 1/2},{-1/4});
        \coordinate (SK) at ({-1/16},{ 1/2},{-3/16});
        \coordinate (SL) at ({ 1/2},{-1/4},{-1/4});
        \coordinate (SM) at ({ 1/2},{-1/4},{-1/16});
        \coordinate (SN) at ({ 1/2},{-3/16},{-1/16});
        \coordinate (SO) at ({ 1/2},{-1/16},{-1/4});
        \coordinate (SP) at ({ 1/2},{-1/16},{-3/16});

        % Hidden edges
        \draw[dotted,very thin] (SA)--(SB);
        \draw[dotted,very thin] (SA)--(SD);
        \draw[dotted,very thin] (SA)--(SL);

        % Visible outer edges
        \draw[very thick]
            (SB)--(SC)--(SE)--(SD)--(SJ)--
            (SO)--(SL)--(SM)--(SH)--cycle;

        % Remaining visible edges
        \draw[very thick] (SC)--(SF)--(SI)--(SH);
        \draw[very thick] (SE)--(SG)--(SF);
        \draw[very thick] (SG)--(SK)--(SJ);
        \draw[very thick] (SI)--(SN)--(SM);
        \draw[very thick] (SK)--(SP)--(SO);
        \draw[very thick] (SN)--(SP);
    \end{scope}

    % ========================================================
    % Arrows: 1 -> 2 -> 3 -> 4 and 1 -> 4
    % ========================================================

    % 1 -> 2
    \draw[diagram arrow]
        (7.75,7)
        -- node[above,font=\small]{$\operatorname{Pyr}(\mathcal{P}^*)$}
        (5.85,7);

    % 2 -> 3
    \draw[diagram arrow]
        (3.2,4.4)
        -- node[left,font=\large]{${\substack{\text{stellar}\\\text{subdivisions}}}$}
        (3.2,2.6);

    % 3 -> 4
    \draw[diagram arrow,<->]
        (5.85,-0.5)
        -- node[above,font=\small]{polar dual}
        (7.75,-0.5);

    % 1 -> 4
    \draw[dashed,thick, ->]
        (9.3,5.2)-- node[left,font=\small]{$\DStell(\mathcal{P}^*)\enspace$}(4.7,0.9);

    \end{tikzpicture}

    \caption{Construction, up to polar duality, of the $3$-dimensional stellahedron from a $2$-simplex.}
    \label{fig:DStell-Delta2-construction}
\end{figure}

The reason for choosing the name ``stellahedral transform'' becomes transparent from the case of a simplex. Up to polar duality, the procedure described in Theorem~\ref{thm:intro-polytopal} sends simplices of dimension $r$ to stellahedra of dimension $r+1$. Figure~\ref{fig:DStell-Delta2-construction} depicts the construction when $r=2$; here we use that the polar dual of a simplex is again a simplex.

One of our main algebraic motivations for the definition of the stellahedral transform comes from Chow theory for Eulerian posets. The Hilbert--Poincar\'e series of matroid Chow rings, augmented Chow rings, and intersection cohomologies have been central objects of study within the last decade (see \cite{elias-proudfoot-wakefield,proudfoot-xu-young,ferroni-matherne-stevens-vecchi,ferroni-matherne-vecchi} and the references therein). In the framework of Kazhdan--Lusztig--Stanley theory, these invariants arise from incidence-algebra constructions associated with the characteristic kernel of a matroid (see \cite{proudfoot-kls,ferroni-matherne-vecchi}). Replacing the characteristic kernel by the Eulerian kernel gives analogous Chow, augmented Chow, and $Z$-polynomials for arbitrary Eulerian posets.

For two of these invariants, a geometric or combinatorial interpretation is already known. The Chow polynomial of an Eulerian poset is the $h$-polynomial of the order complex of its proper part; see \cite[Theorem~5.4]{ferroni-matherne-vecchi}. More recently, Ferroni and Riccardi \cite{ferroni-riccardi}, answering a question of Proudfoot \cite{proudfoot-kls}, showed that the $Z$-polynomial of an Eulerian poset $P$ is the toric $h$-polynomial of the poset of intervals of $P$. Our main result supplies the corresponding interpretation for augmented Chow polynomials.

\begin{theorem}\label{thm:augChow_is_toric-h_of_Stell}
Let $P$ be an Eulerian poset. The right augmented Chow polynomial of $P$ equals the toric $h$-polynomial of the stellahedral transform $\Stell(P)$.
\end{theorem}

Dually, the left augmented Chow polynomial of $P$ is the right augmented Chow polynomial of $P^*$ and hence the toric $h$-polynomial of $\Stell(P^*)$. Thus the stellahedral transform provides a geometric model for both left and right augmented Chow polynomials in the Eulerian setting. In particular, Theorem~\ref{thm:augChow_is_toric-h_of_Stell} answers, in a certain sense, the question raised by Ferroni, Matherne, and Vecchi in \cite[Question~5.9]{ferroni-matherne-vecchi}.

We also study the ordinary Eulerian Chow polynomial under the stellahedral transform. We show that the rank and Chow polynomial of $P$ determine the Chow polynomial of $\Stell(P)$, and give an explicit formula in terms of the zeta polynomial of $P$.

Combining Theorems~\ref{thm:augChow_is_toric-h_of_Stell} and~\ref{thm:intro-polytopal} with the K\"ahler package for the intersection cohomology of polytopes \cite{karu04} immediately implies that, when $P$ is the face poset of a polytope, its right and left augmented Chow polynomials have nonnegative and unimodal coefficients. We establish two complementary strengthenings of this statement: one enlarges the class of posets, while the other strengthens the positivity property.

\begin{theorem}
Let $P$ be a Gorenstein* poset. The right and left augmented Chow polynomials of $P$ have nonnegative and unimodal coefficients.
\end{theorem}

The proof uses Karu's nonnegativity theorem for the $\mathbf{cd}$-index of Gorenstein* posets \cite{karu}, together with explicit universal $\mathbf{cd}$-components for augmented Chow polynomials. This conclusion cannot in general be strengthened to $\gamma$-positivity or real-rootedness: the augmented Chow polynomials of a small Gorenstein* poset already provide a counterexample. For face posets of polytopes, however, the stronger inequalities satisfied by polytopal $\mathbf{cd}$-indices yield the following result.

\begin{theorem}
Let $P$ be the face poset of a polytope. The right and left augmented Chow polynomials of $P$ are $\gamma$-positive.
\end{theorem}

This theorem can be viewed as a polytopal analogue of the matroid result in \cite[Theorem~3.25]{ferroni-matherne-stevens-vecchi}. We further conjecture that, among polytopes of any fixed dimension, the simplex minimizes the $\gamma$-polynomials of the right and left augmented Chow polynomials coefficientwise.

The final part of the paper records several limitations of these positivity phenomena. We construct an Eulerian poset whose ordinary and right augmented Chow polynomials both have negative coefficients, answering a question of Ferroni, Matherne, and Vecchi \cite[Question~5.5]{ferroni-matherne-vecchi}. We also show that the known $\gamma$-positivity of ordinary Chow polynomials of Gorenstein* posets \cite[Theorem~5.7]{ferroni-matherne-vecchi} does not imply log-concavity, even for face posets of polytopes.

\begin{theorem}
There exists a convex $140$-polytope whose face poset has ordinary and right augmented Eulerian Chow polynomials that both fail log-concavity. The chain polynomials of the face poset and of its proper part also fail log-concavity.
\end{theorem}
In particular, none of these polynomials is real-rooted. This gives a negative answer to the question of Brenti and Welker \cite[Question~1]{brenti-welker} on the real-rootedness of chain polynomials of face posets of polytopes, and hence to the more general question of Athanasiadis and Kalampogia-Evangelinou \cite[Question~5.2]{athanasiadis-kalampogia} for lower Eulerian Cohen--Macaulay posets. The proper order complex in this example is a flag PL $139$-sphere. In Remark~\ref{rem:small-gorenstein-chow} we also give a rank-$9$ Gorenstein* poset whose proper order complex is a flag PL $7$-sphere with a non-log-concave $h$-polynomial.

\subsection*{Outline}

The paper is organized as follows. In Section~\ref{sec:preliminaries} we recall the necessary background on incidence algebras, Kazhdan--Lusztig--Stanley functions, and Eulerian and Cohen--Macaulay posets, including the definitions of toric $g$-polynomials and Chow polynomials. In Section~\ref{sec:Stellahedral_transform} we introduce the stellahedral transform and realize it as a Bier-poset construction through an auxiliary chain lift. This yields the preservation of Eulerianity, Cohen--Macaulayness, and the Gorenstein* property. In Section~\ref{sec:geometric_interpretation} we realize the transform geometrically by stellar subdivisions of a pyramid, prove Theorem~\ref{thm:intro-polytopal}, and derive an explicit formula for the $f$-polynomial of the resulting polytope. In Section~\ref{sec:chow-theory} we compute the relevant KLS function of $\Stell(P)$, prove Theorem~\ref{thm:augChow_is_toric-h_of_Stell}, and determine the Eulerian Chow polynomial of $\Stell(P)$ from the rank and Chow polynomial of $P$. In Section~\ref{sec:unimodality-gamma} we use the $\mathbf{cd}$-index to prove unimodality for augmented Chow polynomials of Gorenstein* posets and $\gamma$-positivity in the polytopal case. We also exhibit the failure of $\gamma$-positivity for general Gorenstein* posets. Finally, in Section~\ref{sec:counterexamples} we give counterexamples to coefficientwise nonnegativity for Eulerian Chow and augmented Chow polynomials, and to log-concavity and real-rootedness for ordinary and augmented Chow polynomials of face posets of polytopes.

\section{Preliminaries}\label{sec:preliminaries}

\subsection{Incidence algebras and KLS functions}\label{sec:posetnotions}

Our conventions in this paper largely follow those of \cite{ferroni-matherne-vecchi}. In particular, we will use the letter $P$ to denote a bounded partially ordered set. 
The minimum element of $P$ will be customarily denoted by $\widehat{0}$ whereas the maximum will be denoted by $\widehat{1}$. We say that $P$ is \emph{graded} if all maximal chains in each interval of $P$ have the same length. This implies the existence of a map $\rho:P\to \mathbb{Z}_{\geq 0}$ called \emph{the rank function} of $P$, where $\rho(s)$ records the length of any maximal chain starting at $\widehat{0}$ and ending at $s$. We will customarily write $\rho_{st}$ to denote $\rho(t) - \rho(s)$ for $s\leq t$.

The \emph{incidence algebra} of $P$, denoted by $\mathcal{I}(P)$, is the free $\mathbb{Z}[x]$-module over $\Int(P)$, the set of all closed intervals of $P$. In other words, an element $a\in \mathcal{I}(P)$ associates to each closed interval $[s,t]\in \Int(P)$ a polynomial $a_{st}(x)\in \mathbb{Z}[x]$. When we need to specify the variable $x$, we will write $a_{st}(x)$. The product (also known as convolution) of two elements $a,b\in \mathcal{I}(P)$ is defined via
    \[ (ab)_{st} = \sum_{s\leq w\leq t} a_{sw}\, b_{wt} \qquad \text{ for every $s\leq t$ in $P$}.\]
This product operation makes $\mathcal{I}(P)$ into an associative algebra, having an identity $\delta$, given by $\delta_{st} = 0$ for $s < t$ and $\delta_{ss} = 1$ for all $s\in P$. The $\zeta$-function of $P$ is the element $\zeta\in \mathcal{I}(P)$ such that $\zeta_{st} = 1$ for all $s\leq t$. The \emph{M\"obius function} of $P$ is the element $\mu =\zeta^{-1}$. 

We consider a subalgebra $\mathcal{I}_{\rho}(P)$ given by
    \[ \mathcal{I}_{\rho}(P) := \left\{ a \in \mathcal{I}(P) : \deg a_{st} \leq \rho_{st} \text{ for all $s\leq t$}\right\}.\]
There is an involution $\mathcal{I}_{\rho}(P) \to \mathcal{I}_{\rho}(P)$ that we denote by $a\mapsto a^{\rev}$, and is defined by the following equality:
    \[ (a^{\rev}_{st})(x) = x^{\rho_{st}} a_{st}(x^{-1}).\]
A $P$-kernel is an element $\kappa\in \mathcal{I}_{\rho}(P)$ such that $\kappa_{ss}=1$ for all $s\in P$ and $\kappa^{\rev} = \kappa^{-1}$. 

\begin{theorem}\label{thm:kls_functions}
    Let $\kappa\in \mathcal{I}_{\rho}(P)$ be a $P$-kernel. There exists a unique element $f\in \mathcal{I}_{\rho}(P)$ satisfying the following properties:
    \begin{enumerate}[\normalfont(i)]
        \item \label{it:f-i}$f_{ss}(x) = 1$ for all $s\in P$.
        \item \label{it:f-ii} $\deg f_{st}(x) < \frac{1}{2} \rho_{st}$ for all $s < t$.
        \item \label{it:f-iii} $f^{\rev} = \kappa\cdot f$.
    \end{enumerate}
    Similarly, there exists a unique element $g\in \mathcal{I}_{\rho}(P)$ satisfying the following properties:
    \begin{enumerate}[\normalfont(i')]
        \item \label{it:g-i} $g_{ss}(x) = 1$ for all $s\in P$.
        \item \label{it:g-ii} $\deg g_{st}(x) < \frac{1}{2} \rho_{st}$ for all $s < t$.
        \item \label{it:g-iii} $g^{\rev} = g \cdot \kappa$.
    \end{enumerate}
\end{theorem}

Following \cite[Section~2]{proudfoot-kls}, we refer to $f$ (respectively $g$) as the \emph{right} (respectively \emph{left}) Kazhdan--Lusztig--Stanley (KLS) function associated with $\kappa$. We further define
\[
f_P(x) := f_{\widehat{0}\,\widehat{1}}(x) \quad \text{and} \quad g_P(x) := g_{\widehat{0}\,\widehat{1}}(x),
\]
and call them the \emph{right} and \emph{left} \emph{Kazhdan--Lusztig--Stanley (KLS) polynomials} of $P$, respectively.

\subsection{Eulerian posets}

We will assume that the reader is acquainted with the basic properties of Eulerian posets, and we refer to \cite[Section~3.16]{stanley-ec1} for any undefined terminology. However, we make a brief recapitulation of some essential notions that we will use.

A poset $P$ is said to be \emph{Eulerian} if the M\"obius function satisfies $\mu_{st} = (-1)^{\rho(t) - \rho(s)}$ for every $s\leq t$ in $P$. An equivalent way of stating this property consists in saying that $P$ is Eulerian if and only if every nontrivial interval $[s,t]$ contains the same number of elements of odd rank and even rank. Famous examples of Eulerian posets include face posets of convex polytopes, cell posets of regular CW-spheres, and Bruhat intervals of Coxeter groups.

The following provides a characterization of Eulerian posets in terms of kernels.

\begin{proposition}[{\cite[Proposition~7.1]{Stan-loc}}]
    Let $P$ be a finite graded bounded poset. Then $P$ is Eulerian if and only if the element $\varepsilon\in \mathcal{I}_{\rho}(P)$ given by $\varepsilon_{st}(x) = (x-1)^{\rho_{st}}$ is a $P$-kernel.
\end{proposition}
We will refer to $\varepsilon$ as the \emph{Eulerian $P$-kernel}. 
The left KLS polynomial $g_P(x)$ arising from the $P$-kernel $\varepsilon$ in an Eulerian poset $P$ is commonly known as the \emph{toric $g$-polynomial of $P$}. It is not hard to see that the right KLS polynomial $f_P(x)$ arising from $\varepsilon$ equals the toric $g$-polynomial of the dual poset $P^*$.
Define the \emph{reduced Eulerian $P$-kernel} $\overline{\varepsilon}\in \mathcal{I}_{\rho}(P)$ as $$\overline{\varepsilon}_{st}(x) = \begin{cases} (x-1)^{\rho_{st} - 1} & s < t,\\ -1 & s = t,\end{cases}$$
and consider $\H := -\left(\overline{\varepsilon}\right)^{-1}$.
This element of the incidence algebra of $P$ is called the \emph{Eulerian Chow function}. From this, define $F = \H \cdot f^{\rev}$ and $G = g^{\rev}\cdot \H$, called \emph{the right (resp. left) augmented Eulerian Chow function}. For more details about Chow functions, we refer to \cite{ferroni-matherne-vecchi}. 

Finally, the \emph{toric $h$-polynomial} of $P$ is defined as
    \[ h_P(x) = \frac{g_P^{\rev}(x) - g_P(x)}{x-1}.\]
More generally, 
\[h_{st}(x) = \frac{g_{st}^{\rev}(x) -g_{st}(x)}{x-1}\] for each $s<t$ in $P$.
In Stanley's book \cite[Section~3.16]{stanley-ec1} this polynomial is denoted by the letter $f$, but recall that here we reserve that notation for the right KLS polynomial.  The work of Bayer and Ehrenborg \cite{bayer-ehrenborg} provides a thorough study of toric $g$-polynomials and toric $h$-polynomials of more general (not necessarily Eulerian) posets. For a bounded Eulerian poset $P$, we write
\[
\H_P(x):=\H_{\widehat{0}\widehat{1}}(x),\qquad
F_P(x):=F_{\widehat{0}\widehat{1}}(x),\qquad
G_P(x):=G_{\widehat{0}\widehat{1}}(x).
\]

\subsection{Cohen--Macaulay posets}
Recall that to every poset $P$ we may associate a simplicial complex $\Delta(P)$, called the \emph{order complex of $P$}. The faces of $\Delta(P)$ correspond to chains of elements in $P$.  If $P$ is bounded, we write $\overline P:=P\setminus\{\widehat0,\widehat1\}$ for its proper part, and call $\Delta(\overline P)$ the \emph{proper
order complex of $P$}.

Like any other simplicial complex, $\Delta(P)$ admits a geometric realization that we will denote $|\Delta(P)|$.
For a poset $P$, we call $J$ an open interval of $P$ if it is of the form of one of the following posets $(s,t) = \{w\in P : s < w < t\}$, $P_{<t} = \{w\in P : w < t\}$, or $P_{>s} = \{w\in P : s < w\}$ for some $s<t$ in $P$. Note that every interval of a graded poset is itself a graded poset. By definition, we say that $P$ is \emph{Cohen--Macaulay} (over $\mathbb{Q}$) if the rational reduced homology groups of the order complex of every open interval $J \subseteq P$ satisfy
    \[ \widetilde{H}_i(\Delta(J)) = 0 \qquad \text{ for all $i < \dim \Delta(J)$}.\]

In other words, the (reduced) homology of every open interval $J$ must be concentrated in dimension $\dim \Delta(J)$. The class of Cohen--Macaulay posets comprises a number of well-studied families, such as geometric lattices, distributive lattices, posets that are EL-shellable, etc. It is worth noting that Cohen--Macaulayness is a topological property, that is, if $P_1$ and $P_2$ are posets and there is a homeomorphism $|\Delta(P_1)|\approx |\Delta(P_2)|$ then $P_1$ is Cohen--Macaulay if and only if $P_2$ is Cohen--Macaulay.

Furthermore, we say that a poset is \textit{Gorenstein*} if it is both Eulerian and Cohen--Macaulay. In order to prove that a poset is Cohen--Macaulay, a very useful result is the following.

\begin{theorem}[{\cite[Corollary~10.12]{bjorner_topological_methods}}] \label{thm:poset_map_homotopically_equivalent}
Let $P$ be a poset and $f: P \to P$ an order-preserving function. If $f(x) \geq x$ for every $x \in P$, then $|\Delta(P)|$ is homotopically equivalent to $|\Delta(f(P))|$. 
\end{theorem}

\section{Stellahedral transform}\label{sec:Stellahedral_transform}
In this section, we introduce the \textit{stellahedral transform} of a partially ordered set. We demonstrate that it can be realized as a Bier poset (see Definition~\ref{def:bier-poset} for the meaning of this term), and we subsequently prove that it preserves some fundamental poset properties.

\begin{definition}
Let $P$ be a poset with minimum element ${\zero}$. We define the \textit{stellahedral transform} of $P$ as the poset
\[
\Stell(P)= \left\{ (c,y):  c=\left\{{\zero} =c_0 < c_1 < \dots < c_k\right\} \text{ chain in $P$},\,y \in P,\, c_k \leq y\right\} \sqcup \left\{{\one}_{\Stell}\right\}
\]
where the partial order is given by imposing ${\one}_{\Stell}$ as the maximum element and taking 
\[
(c,y) \le (c',y') \iff c \subseteq c' \text{ and } y \ge y' \text{ in $P$}.
\]
\end{definition}

This poset transformation preserves several features of the poset one feeds into it. The following lemma shows that if a poset $P$ is bounded and graded with rank $n$, its stellahedral transform is bounded and graded with rank $n+1$.

\begin{lemma}
\label{lm: rank}
Let $P$ be a graded bounded poset with ${\zero}$ and $\one$ respectively the minimum element and the maximum element.
Then, the poset $\Stell(P)$ is:
\begin{itemize}
\item bounded with minimum element $(\{{\zero}\}, {\one})$ and maximum element ${\one}_{\Stell}$;
\item graded by the rank function 
\begin{align*}
\rho_{(c,y),(c',y')} &= \rho_{y'y} + |c'| - |c|, \\
\rho_{(c,y),{\one}_{\Stell(P)}}  &= \rho_{{\zero}y}-|c|+2  
\end{align*}
for every $(c,y) \le (c',y') \in \Stell(P)$.
\end{itemize}
\end{lemma}

\begin{proof}
By construction, $\Stell(P)$ is bounded by $(\{{\zero}\}, {\one}) $ and $\one_{\Stell}$. 
To show that $\Stell(P)$ is graded by the rank function described in the statement, consider a covering relation $(c,y) \lessdot (c',y')$ for some $(c,y),\,(c',y') \in \Stell(P)$. We have that either $y$ covers $y'$ in $P$ and $c= c'$, or $y = y'$ and $c'$ covers $c$ in the poset of chains, which means $c \subseteq c'$ and  $|c'| = |c| + 1$.
In both cases, $\rho_{y'y} + |c'| - |c|=1$. For a top cover, $|c|=\rho_{{\zero}y}+1$. This proves the stated rank formulas.%By induction on the length of maximal chains, $\Stell(P)$ is a graded poset with the stated rank function.
\end{proof}

The stellahedral transform of a poset might exhibit a subtle behavior even when the starting poset is quite simple: the condition $c_k \leq y$ interlaces the internal structure of $P$ with that of its chains, making $\Stell(P)$ significantly different from $P$, its order complex $\Delta(P)$, and  their direct product.

\begin{example}
Let $ A_n$ be the antichain of $n$ elements. Call $P$ the poset obtained by adjoining a minimum element ${\zero}$ to $A_n$. Then $\Stell(P)$ has the Hasse diagram shown in Figure \ref{fig:antichain_st}.
\begin{figure}[ht]
    \centering
    \begin{tikzpicture}[baseline=(current bounding box.center),
      scale=0.85,auto=center,
      every node/.style={circle,scale=0.8,fill=black,inner sep=2.7pt}]
      \tikzstyle{edges}=[thick]

      % --- LEFT DIAGRAM: P_n ---
      \node[label=below:{ \itshape ${\zero}$}] (cero) at (0,0) {};
      \node[label=above:{ \itshape $x_1$}] (x1) at (-1.5,2) {};
      \node[label=above:{ \itshape $x_2$}] (x2) at (-0.5,2) {};
      \node[fill=none,draw=none,scale=1.2,inner sep=0] (dots1) at (0.5,2) {$\dots$};
      \node[label=above:{ \itshape $x_n$}] (xn) at (1.5,2) {};

      % Edges
      \draw[edges] (cero) -- (x1);
      \draw[edges] (cero) -- (x2);
      \draw[edges] (cero) -- (xn);
    \end{tikzpicture}
    \hspace{1.5cm}
    \begin{tikzpicture}[baseline=(current bounding box.center),
      scale=0.7,auto=center,
      every node/.style={circle,scale=0.8,fill=black,inner sep=2.7pt}]
      \tikzstyle{edges}=[thick];

      % --- RIGHT DIAGRAM: St(P_n) ---
      
      % RANK 0: Bottom elements
      \node[label={[label distance=-2pt]below:{\small  \itshape $\{{\zero}\} x_1\,$}}] (b1) at (-2,0) {};
      \node[label={[label distance=-2pt]below:{\small  \itshape $\{{\zero}\} x_2\,$}}] (b2) at (0,0) {};
      \node[fill=none,draw=none,scale=1.2,inner sep=0] (bdots) at (2,0) {\large $\dots$};
      \node[label={[label distance=-2pt]below:{\small  \itshape $\{{\zero}\} x_n\,$}}] (bn) at (4,0) {};

      % RANK 1: Middle elements
      \node[label=left:{ \small \itshape $\{{\zero}\} {\zero}\,$}] (m0) at (-4,2.5) {};
      \node[label=right:{ \small  \itshape $\{{\zero}, x_1\} x_1\,$}] (m1) at (-2,2.5) {};
      \node[label=right:{  \small \itshape $\{{\zero}, x_2\} x_2\,$}] (m2) at (0,2.5) {};
      \node[fill=none,draw=none,scale=1.2,inner sep=0] (mdots) at (2,1.25) {\large $\dots$};
      \node[fill=none,draw=none,scale=1.2,inner sep=0] (mdots) at (2.65,2.5) {\large $\dots$};
      \node[label=right:{ \small  \itshape $\{{\zero}, x_n\} x_n\,$}] (mn) at (4,2.5) {};

      % RANK 2: Top element
      \node[label=above:{ \small \itshape  ${\one}_{\Stell}$}] (uno) at (0,5) {};

      % Edges from rank 0 to rank 1
      \draw[edges] (b1) -- (m0);
      \draw[edges] (b2) -- (m0);
      \draw[edges] (bn) -- (m0);
      
      \draw[edges] (b1) -- (m1);
      \draw[edges] (b2) -- (m2);
      \draw[edges] (bn) -- (mn);

      % Edges from rank 1 to rank 2
      \draw[edges] (m0) -- (uno);
      \draw[edges] (m1) -- (uno);
      \draw[edges] (m2) -- (uno);
      \draw[edges] (mn) -- (uno);

      % Label under diagram
      %\node[fill=none,draw=none, scale=1.2] at (0,-1.75) {$\Stell(P)$};
    \end{tikzpicture}
    \caption{On the left, the Hasse diagram of $P = A_n \sqcup \{{\zero}\}$; on the right, the Hasse diagram of $\Stell(P)$.}
    \label{fig:antichain_st}
\end{figure}
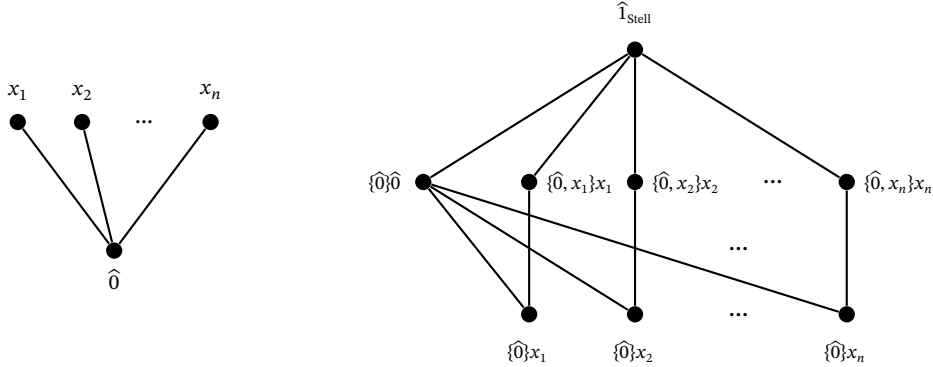
\end{example}
\begin{example}
Let $P = C_n = \{0,1,\dots,n\}$ be the chain of rank $n$. Then $\Stell(P)\setminus\{\one_{\Stell}\}$ is isomorphic to the following subposet of $B_n \times C_n$
\[
\{ (S,x) \in B_n \times C_n \mid S \neq \emptyset \Rightarrow x \leq n-\max(S) \}
\]
where $B_n = \mathcal{P}([n])$ is the boolean lattice of rank $n$.
\begin{figure}[ht!]
    \centering
    \begin{tikzpicture}[baseline=(current bounding box.center),
      scale=0.85,auto=center,
      every node/.style={circle,scale=0.8,fill=black,inner sep=2.7pt}]
      \tikzstyle{edges}=[thick];

      % RANK 4: Maximal element
      \node[label=above:{\footnotesize $\widehat{1}_{\mathrm{St}}$}] (n4) at (0,6) {};

      % RANK 3
      \node[label=left:{\footnotesize $(\emptyset, 3)$}] (n3_0) at (-4.5,4.5) {};
      \node[label=left:{\footnotesize $(\{1\}, 2)$}] (n3_1) at (-1.5,4.5) {};
      \node[label=right:{\footnotesize $(\{1,2\}, 1)$}] (n3_12) at (1.5,4.5) {};
      \node[label=right:{\footnotesize $(\{1,2,3\}, 0)$}] (n3_123) at (4.5,4.5) {};

      % RANK 2
      \node[label=left:{\footnotesize $(\emptyset, 2)$}] (n2_0) at (-5,3) {};
      \node[label=left:{\footnotesize $(\{1\}, 1)$}] (n2_1) at (-3,3) {};
      \node[label=left:{\footnotesize $(\{2\}, 1)$}] (n2_2) at (-1,3) {};
      \node[label=right:{\footnotesize $(\{1,2\}, 0)$}] (n2_12) at (1,3) {};
      \node[label=right:{\footnotesize $(\{1,3\}, 0)$}] (n2_13) at (3,3) {};
      \node[label=right:{\footnotesize $(\{2,3\}, 0)$}] (n2_23) at (5,3) {};

      % RANK 1
      \node[label=left:{\footnotesize $(\emptyset, 1)$}] (n1_0) at (-4.5,1.5) {};
      \node[label=left:{\footnotesize $(\{1\}, 0)$}] (n1_1) at (-1.5,1.5) {};
      \node[label=right:{\footnotesize $(\{2\}, 0)$}] (n1_2) at (1.5,1.5) {};
      \node[label=right:{\footnotesize $(\{3\}, 0)$}] (n1_3) at (4.5,1.5) {};

      % RANK 0
      \node[label=below:{\footnotesize $(\emptyset, 0)$}] (n0) at (0,0) {};

      % --- EDGES ---
      
      % R0 to R1
      \draw[edges] (n0) -- (n1_0);
      \draw[edges] (n0) -- (n1_1);
      \draw[edges] (n0) -- (n1_2);
      \draw[edges] (n0) -- (n1_3);

      % R1 to R2
      % from (\emptyset, 1)
      \draw[edges] (n1_0) -- (n2_0);
      \draw[edges] (n1_0) -- (n2_1);
      \draw[edges] (n1_0) -- (n2_2);
      % from ({1}, 0)
      \draw[edges] (n1_1) -- (n2_1);
      \draw[edges] (n1_1) -- (n2_12);
      \draw[edges] (n1_1) -- (n2_13);
      % from ({2}, 0)
      \draw[edges] (n1_2) -- (n2_2);
      \draw[edges] (n1_2) -- (n2_12);
      \draw[edges] (n1_2) -- (n2_23);
      % from ({3}, 0)
      \draw[edges] (n1_3) -- (n2_13);
      \draw[edges] (n1_3) -- (n2_23);

      % R2 to R3
      % from (\emptyset, 2)
      \draw[edges] (n2_0) -- (n3_0);
      \draw[edges] (n2_0) -- (n3_1);
      % from ({1}, 1)
      \draw[edges] (n2_1) -- (n3_1);
      \draw[edges] (n2_1) -- (n3_12);
      % from ({2}, 1)
      \draw[edges] (n2_2) -- (n3_12);
      % from ({1,2}, 0)
      \draw[edges] (n2_12) -- (n3_12);
      \draw[edges] (n2_12) -- (n3_123);
      % from ({1,3}, 0)
      \draw[edges] (n2_13) -- (n3_123);
      % from ({2,3}, 0)
      \draw[edges] (n2_23) -- (n3_123);

      % R3 to R4 (Global Max)
      \draw[edges] (n3_0) -- (n4);
      \draw[edges] (n3_1) -- (n4);
      \draw[edges] (n3_12) -- (n4);
      \draw[edges] (n3_123) -- (n4);

    \end{tikzpicture}
    \caption{Hasse diagram of $\Stell(C_3)$.}
    \label{fig:st_c3}
\end{figure}
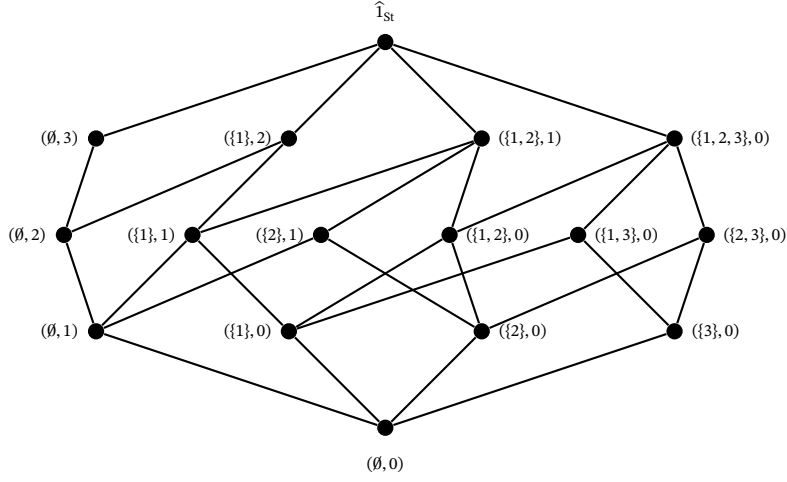
\end{example}

The construction of the stellahedral transform of a poset can be reinterpreted using the concept of a \textit{Bier poset}, originally defined in \cite{bjorner_bier_2005}.  
\begin{definition}\label{def:bier-poset}
Let $P$ be a bounded poset of finite length and let $I \subseteq P$ be a proper ideal. Then the poset $\Bier(P,I)$ is defined as
\[
\Bier(P,I) = \{ [x,y] \subseteq P \mid x \in I,\, y \not\in I\} \sqcup \{{\one}_{\Bier}\}
\]
by imposing ${\one}_{\Bier}$ as the maximum and ordering the other elements by reverse inclusion.
\end{definition}

Given a bounded poset $P$, let
$I_P := \left\{ \{{\zero} = y_0 < y_1 < \dots < y_r\} \text{ chain in $P$} \right\}$ be the set of all chains of $P$ anchored at the bottom element $\widehat{0}$. This is naturally a poset using the partial order given by set inclusion. We now define an intermediate construction that will turn out to be useful in our discussion. 

\begin{definition}
    Let $P$ be a bounded poset. The \textit{chain lift} of $P$ is the poset on the set $P^{\uparrow} := I_P \sqcup P$, with the order given by
    \[
p \leq_{{\uparrow}} q \iff 
\begin{cases}
p \subseteq q  & p,q \in I_P, \\
p \leq q & p,q \in P, \\
\max p \leq q & p \in I_P,\,q \in P.
\end{cases}
\]
\end{definition}

Notice that if $P$ is a graded poset, then the chain lift $P^{\uparrow}$ is graded by 
\[
\rho^{\uparrow}_{p,q} = \begin{cases} |q|-|p| &\text{if }p,q \in I_P, \\
\rho_{p,q} &\text{if }p,q \in P, \\
\rho_{{\zero},q} - |p|+2 & \text{if }p \in I_P \text{ and }q \in P.
\end{cases}
\]
From this setup, we obtain the following identity
\begin{equation}\label{eq:St_Bier}
\Stell(P) = \Bier(P^{\uparrow},I_P). 
\end{equation}
We now prove some properties of $P^{\uparrow}$. Then, using known results on Bier posets, it will be possible to transfer them to the stellahedral transform $\Stell(P)$.
\begin{lemma}
If $P$ is an Eulerian poset, then $P^{\uparrow}$ is also Eulerian.
\end{lemma}
\begin{proof}
Let $J= [p,q] \subseteq P^{\uparrow}$ be an interval. We prove that $[p,q]$ has the same number of elements of even and odd rank when $p<q$. We divide the proof into four cases.
\begin{enumerate}
\item If $p,q \in P \subseteq P^{\uparrow}$, then $J \subseteq P$ is Eulerian because $P$ is Eulerian.

\item If $p,q \in I_P$, then $J$ is isomorphic to a boolean lattice, which is Eulerian.

\item If $p \in I_P,\, q \in P$ and $\max p < q$, consider the map 
\begin{align*}
f:J\cap I_P &\to J \cap I_P \\
S &\mapsto S \cup \{q\} \quad \text{if $q \not \in S$} \\
S &\mapsto S \setminus \{q\} \quad \text{ if $q  \in S$}.
\end{align*}
Notice that $f$ is a rank parity reversing involution, hence $J \cap I_P$ has the same number of elements of even and odd rank. Since $J \cap P \cong [\max p, q]$, $J$ has also the same number of elements of even and odd rank.

\item If $p \in I_P$, $q \in P$ and $\max p = q$, then 
\[
J = \{ c \in I_{[{\zero},q]} \mid p \subseteq c \} \sqcup \{q\}.
\]
Let $C_q$ be the poset of nonempty chains of $({\zero},\,q)$. If $|p|>2$,  $J$ is isomorphic to the interval $[p\setminus \{{\zero},\,q\},\,\one]_{C_q \sqcup \{{\zero},\one\}}$. Similarly, if $|p| = 2$, the poset $J$ is isomorphic to $C_q \sqcup \{{\zero},\one\}$. In both cases, we may conclude by Exercise 3.141(c) of \cite{stanley-ec1} that $J$ is Eulerian. Otherwise, if $|p| = 1$ (equiv.  ${\zero} = q$), then $J$ is the unique Eulerian poset of rank $1$.\qedhere
\end{enumerate}
\end{proof}

\begin{lemma}
If $P$ is bounded and Cohen--Macaulay, then $P^{\uparrow}$ is also Cohen--Macaulay.
\end{lemma}
\begin{proof}
Let $J$ be an open interval of $P^{\uparrow}$. We need to prove that $\tilde{H}_i(\Delta(J),\mathbb{Q})=0$ for all $i< \dim \Delta(J)$, hence we may assume without loss of generality that $\dim \Delta(J)\neq 0$. We divide the proof into five possible cases. 
\begin{enumerate}[\normalfont (i)]
\item If $J \subseteq P$, it satisfies the CM condition by hypothesis.
\item If $J = (p,q)$ where $p,q \in I_P$ , then $J$ is isomorphic to an open interval of the boolean lattice, which is Cohen--Macaulay. 
\item If $J = (p,q)$ where $p=\{{\zero} = p_0< p_1 < \dots < p_{r-1}< p_r \} \in I_P,\, q \in P$ and $p_{r} = q$, we consider the face poset $C$ of the complex $\Delta({\zero},q)$ and its element $p' = \{ p_1 < \dots < p_{r-1}\} \in C$. We know that $C$ is Cohen--Macaulay, since being Cohen--Macaulay is a topological property and the open interval $(\oldhat0,q) \subseteq P$ is Cohen--Macaulay. We observe that the poset $J$ is isomorphic to $C_{>p'}$, so it satisfies the CM condition.
\item If $J = (p,q)$ where $p=\{{\zero} = p_0 < p_1 < \dots < p_r\} \in I_P,\, q \in P$ and $p_r < q$,  we take the subposets 
\begin{align*}
U &= \{ j \in J \cap I_P \mid \max(j) \leq q \},\\ 
V &= \{ j \in J \cap I_P \mid \max(j) < q \} \cup \left[p_r,q\right)_P.
\end{align*}
Notice that $U \cap V = \{ j \in J \cap I_P \mid \max(j) < q \} = (p\setminus\{{\zero}\},{\one})_{\Delta(({\zero},q))}$, which is Cohen--Macaulay by the same argument of point (iii). We may construct a pair of functions
\begin{align*}
    f:U &\to U \\
    x &\mapsto x \cup \{q\},
\end{align*}
and
\begin{align*}
    g:V &\to V \\
    x &\mapsto \max(x) ~ \quad \text{ if $x \in I_P$}, \\
    x &\mapsto x \qquad \qquad \text{if $x \in P$}.
\end{align*}
Since $f$ and $g$ are order-preserving maps such that $f(x) \geq x$ and $g(y) \geq y$ for all $x \in U$ and $y \in V$, by Theorem \ref{thm:poset_map_homotopically_equivalent} we have
\[
\Delta(U) \simeq \Delta(f(U)), \quad \Delta(V) \simeq \Delta(g(V)) = \Delta([p_r,q)).
\]
If a poset has a minimum element or a maximum element, it is always contractible since its order complex is homotopically equivalent to a pyramid. Hence $\Delta(U)$ and $\Delta(V)$ are contractible, in fact $p\cup\{q\}$ is the minimum of $f(U)$ and $p_r$ is the minimum of $g(V)$. Since $\Delta(U) \cap \Delta(V) = \Delta(U\cap V)$ and $\Delta(J) = \Delta(U) \cup \Delta(V)$, it follows from the Mayer--Vietoris Theorem that for $0 \leq i < \dim \Delta(J)-1$ we have
\[ 
\tilde{H}_{i+1}(\Delta(J),\mathbb{Q}) =  \tilde{H}_{i}(\Delta((p\setminus\{{\zero}\},{\one})_{\Delta(({\zero},q))}),\mathbb{Q}) = 
0
\]
and $\tilde{H}_{0}(\Delta(J),\mathbb{Q}) =0$. 
\item If $J=(P^{\uparrow})_{<p}$  or $J=(P^{\uparrow})_{>p}$ for some $p \in P^{\uparrow}$, then $J$ contains either a minimum element or a maximum element. Consequently, $\Delta(J)$ is contractible and hence it has trivial reduced homology groups.\qedhere 
\end{enumerate}
\end{proof}
Using the identity in equation~\eqref{eq:St_Bier}, we may apply Corollary 2.5 and Theorem 3.1 from \cite{bjorner_bier_2005} to the stellahedral transform, obtaining the following result.

\begin{theorem}
The stellahedral transform preserves the property of being Eulerian, Cohen--Macaulay and Gorenstein*.
\end{theorem}

As we will see below, in Corollary \ref{cor:stellahedral_poset_polytope}, the stellahedral transform possesses an even more striking feature: it also preserves the property of being the face poset of a polytope. In order to prove this, it is necessary to provide a geometric interpretation of the transform, which will be given in Section \ref{sec:geometric_interpretation}.

\section{Geometric interpretation}\label{sec:geometric_interpretation}
In this section we provide a geometric interpretation of the stellahedral transform. As anticipated in Section \ref{sec:Stellahedral_transform}, the transform preserves the property of being the face poset of a polytope. The following discussion aims to construct the operation on polytopes that corresponds to the stellahedral transform on their respective face posets. 

Given an $n$-dimensional polytope $\mathcal{P} \subseteq \mathbb{R}^n$ and a face $F \leq \mathcal{P}$, we say that a point $y \in \mathbb{R}^n$ is \textit{beyond} $F$ if, for every facet $G$ of $\mathcal{P}$ containing $F$, the point $y$ lies in the open half-space determined by the hyperplane $\operatorname{aff}(G)$ opposite to $\mathcal{P}$, and for every facet $G$ not containing $F$, the point $y$ lies in the open half-space containing the interior of $\mathcal{P}$ with respect to $\operatorname{aff}(G)$.

We recall that the \textit{stellar subdivision} of a polytope $\mathcal{P}$ at a non-empty proper face $F$ is the polytope defined by taking the convex hull of $\mathcal{P}$ and $v$, where $v \in \R^n$ is any point that is beyond $F$. Although the resulting polytope depends on the choice of $v$, its face poset does not. Combinatorially, it is obtained by removing the proper faces containing $F$ and replacing them with the closure of the set 
\[
\left\{ v \star G \mid \exists H \in \mathcal{F}(\mathcal{P}) \setminus \{\mathcal{P}\}: F, G  \leq H,\, F \not\subseteq G \right\}.
\]

\begin{definition}
    Let $\mathcal{P}$ be a polytope. Up to combinatorial equivalence, we call the \textit{dual stellahedral transform} of $\mathcal{P}$ the polytope $\DStell(\mathcal{P})$ that is constructed by first taking the pyramid over $\mathcal{P}$ and then applying a stellar subdivision to every proper face strictly containing the apex of the pyramid, in order of decreasing dimension.
\end{definition}

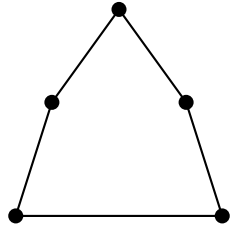
\begin{figure}[ht!]
    \centering
    \begin{tikzpicture}[baseline=(current bounding box.center), 
      scale=2.73, 
      vertex/.style={circle,fill=black,inner sep=2pt}]
    \tikzstyle{edges}=[thick];

    % ---- Coordinates ----
    \coordinate (v1) at (0.5, 1.0);
    \coordinate (v2) at (0.175, 0.55);
    \coordinate (v3) at (0.0, 0.0);
    \coordinate (v4) at (1.0, 0.0);
    \coordinate (v5) at (0.825, 0.55);

    \draw[edges] (v1) -- (v2) -- (v3) -- (v4) -- (v5) -- cycle;

    % ---- Draw Vertices ----
    \node[vertex] at (v1) {};
    \node[vertex] at (v2) {};
    \node[vertex] at (v3) {};
    \node[vertex] at (v4) {};
    \node[vertex] at (v5) {};

    \end{tikzpicture}
    \caption{Dual stellahedral transform of the $1$-simplex.}
    \label{fig:2d_polytope}
\end{figure}
The following theorem connects the previous construction to the stellahedral transform of posets.
\begin{theorem}\label{thm:face_lattice_St(P)}
Let $\mathcal{P}$ be a polytope and $S$ be the poset of pairs $(\mathcal{C}, F)$ where
\begin{enumerate}[\normalfont(i)]
    \item $F \in \mathcal{F}(\mathcal{P})$ is a face of $\mathcal{P}$,
    \item $\mathcal{C} = \left\{F_1 \subsetneq F_2 \subsetneq \dots \subsetneq F_m \right\}$ is a (possibly empty) strict chain in $\mathcal{F}(\mathcal{P}) \setminus \{\mathcal{P}\}$,
    \item If $\mathcal{C} \neq \emptyset$, then $F \subseteq F_1$,
\end{enumerate}
  such that 
  \[
  (\mathcal{C}, F) \le ( \mathcal{C}', F') \iff F \subseteq F' \text{ and } \mathcal{C} \subseteq \mathcal{C'}.
  \]
  If we add a maximum element $\one_S$ to $S$, the face poset of $\DStell(\mathcal{P})$ is isomorphic to $S \sqcup \{\one\}$. In particular, there is an isomorphism given by the function 
  \begin{align*}
        \Phi: S \sqcup \{\one_S\} &\to  \mathcal{F}(\DStell(\mathcal{P}))   \\
        (\mathcal{C}, F) &\mapsto \operatorname{conv}(F,v_{F_1},\dots,v_{F_m}) \\
        \one_S &\mapsto \one_{\DStell(\mathcal{P})}
    \end{align*}
    where $\mathcal{C}=\left\{F_1 \subsetneq \dots \subsetneq F_m\right\}$ and $F$ is seen as embedded in the base of $\DStell(\mathcal{P})$. Here, $v_{\oldhat{0}}$ denotes the apex of the pyramid over $\mathcal{P}$, while $v_G$ denotes the vertex introduced when performing the stellar subdivision of the face $ v_{\oldhat{0}} \star G \leq \operatorname{Pyr}(\mathcal{P})$ for every nonempty proper face $G$ of $\mathcal{P}$.
\end{theorem}
    The map $\Phi$ induces a labeling of the faces of the dual stellahedral transform. In particular, the vertices of $\DStell(\mathcal{P})$ and the rank-$1$ elements of $S$ are such that
    \begin{itemize}
        \item the vertices $v$ of $\mathcal{P} \subseteq \DStell(\mathcal{P})$ are in bijection with the pairs $(\emptyset, v)$ for  all vertices $v \in \mathcal{P}$;
        \item the apex $v_{\oldhat{0}}$ of the pyramid over $\mathcal{P}$ is the image of the element $(\{{\zero}\},{\zero})$;
        \item the vertices $v_F$ are in bijection with the pairs $(\{F\},{\zero})$ for all faces $F \in \mathcal{F}(\mathcal{P}) \setminus\{\mathcal{P}\}$.
    \end{itemize}   
    If $\mathcal{P}$ is a square, we may see in Figure \ref{figure:vertices_St(P)} its vertex labels.
    
    \begin{figure}[ht!]
    \centering
    \begin{tikzpicture}
    \begin{scope}[line width =1.2pt,scale=0.8,
    vertex/.style={circle,fill,inner sep=2pt}]
    \tikzstyle help lines=[color=blue!50,very thin]
    \draw (0,0) -- (5,0) -- (5,5) -- (0,5) -- (0,0);
    % Symbols
    \node at (7.35, 2.5) {{\huge $\longrightarrow$}};

    % FIRST POLYTOPE
    \coordinate (Q1) at (5,0);
    \coordinate (Q2) at (5,5);
    \coordinate (Q3) at (0,5);
    \coordinate (Q4) at (0,0);

    \coordinate (E12) at (5,2.5);
    \coordinate (E23) at (2.5,5);
    \coordinate (E34) at (0,2.5);
    \coordinate (E14) at (2.5,0);

    \node[vertex, label=below:{\contour{white}{$v_1$}}] at (Q1) {};
    \node[vertex, label=above:{\contour{white}{$v_2$}}] at (Q2) {};
    \node[vertex, label=above:{\contour{white}{$v_3$}}] at (Q3) {};
    \node[vertex, label=below:{\contour{white}{$v_4$}}] at (Q4) {};

    \node[label=right:{\contour{white}{$e_{12}$}}] at (E12) {};
    \node[label=above:{\contour{white}{$e_{23}$}}] at (E23) {};
    \node[label=left:{\contour{white}{$e_{34}$}}] at (E34) {};
    \node[label=below:{\contour{white}{$e_{14}$}}] at (E14) {};
    \end{scope}
    \tdplotsetmaincoords{70}{108} % Adjust view angle
    \begin{scope}[tdplot_main_coords, 
    scale=2.8, 
    xshift=3.5cm,
    vertex/.style={circle,fill,inner sep=2pt}]
    \tikzstyle help lines=[color=blue!50, thin]

    % ---- Coordinates ----
    % Base vertices
    \coordinate (B1) at (1.0, 1.0, 0.0);
    \coordinate (B2) at (-1.0, 1.0, 0.0);
    \coordinate (B3) at (-1.0, -1.0, 0.0);
    \coordinate (B4) at (1.0, -1.0, 0.0);

    % Apex
    \coordinate (A)  at (0.0, 0.0, 2.0);

    % Stellar subdivision of the faces (Facets)
    \coordinate (F1) at (0.0, 0.8666666667, 0.8666666667);
    \coordinate (F2) at (-0.8666666667, 0.0, 0.8666666667);
    \coordinate (F3) at (0.0, -0.8666666667, 0.8666666667);
    \coordinate (F4) at (0.8666666667, 0.0, 0.8666666667);

    % Stellar subdivision of the edges
    \coordinate (E1) at (0.55, 0.55, 1.1);
    \coordinate (E2) at (-0.55, 0.55, 1.1);
    \coordinate (E3) at (-0.55, -0.55, 1.1);
    \coordinate (E4) at (0.55, -0.55, 1.1);

    % Red Face Label
    \coordinate (RF) at (1, -0.2, -0.1);

    % ---- Base Edges ----
    % Base back edges
    \draw[thick] (B4) -- (B1) -- (B2);
    \draw[dotted] (B2) -- (B3) -- (B4);

    % Hidden Face/Edge nodes to Base
    \draw[dotted] (B2) -- (F2);
    \draw[dotted] (B3) -- (F2);
    \draw[dotted] (B3) -- (F3);
    \draw[dotted] (B3) -- (E3);
    % Visible Face/Edge nodes to Base
    \draw[thick] (B1) -- (F1);
    \draw[thick] (B1) -- (F4);
    \draw[thick] (B1) -- (E1);
    \draw[thick] (B2) -- (F1);
    \draw[thick] (B2) -- (E2);    
    \draw[thick] (B4) -- (F3);
    \draw[thick] (B4) -- (F4);
    \draw[thick] (B4) -- (E4);

    % Hidden Face nodes to Edge nodes
    \draw[dotted] (F2) -- (E2);
    \draw[dotted] (F2) -- (E3);
    \draw[dotted] (F3) -- (E3);
    % Visible Face nodes to Edge nodes
    \draw[thick] (F1) -- (E1);
    \draw[thick] (F1) -- (E2);
    \draw[thick] (F3) -- (E4);
    \draw[thick] (F4) -- (E4);
    \draw[thick] (F4) -- (E1);

    % Hidden Edge/Face nodes to Apex
    \draw[dotted] (F2) -- (A);    
    \draw[dotted] (E3) -- (A);    
    % Visible Edge/Face nodes to Apex
    \draw[thick] (E1) -- (A);
    \draw[thick] (F1) -- (A);
    \draw[thick] (E2) -- (A);
    \draw[thick] (F3) -- (A);
    \draw[thick] (E4) -- (A);
    \draw[thick] (F4) -- (A);

    % ---- Draw Vertices ----
    \node[vertex, label=below:{\contour{white}{\small $ \emptyset\,  v_1$}}] at (B1) {};
    \node[vertex, label=below:{\contour{white}{\small $\emptyset\, v_2$}}] at (B2) {};
    \node[vertex, label=below:{\contour{white}{\small $\emptyset\, v_3$}}] at (B3) {};
    \node[vertex, label=below:{\contour{white}{\small $\emptyset\, v_4$}}] at (B4) {};
    
    \node[vertex, label=above:{\small $ \{{\zero}\}{\zero}$}] at (A) {};
    
    \node[vertex, label=right:{\contour{white}{\small $\left\{ e_{12}\right\}{\zero}$}}] at (F1) {};
    \node[vertex, label=above:{\contour{white}{\small $\left\{ e_{23}\right\}{\zero}$}}] at (F2) {};
    \node[vertex, label=left:{\contour{white}{\small $\left\{ e_{34}\right\}{\zero}$}}] at (F3) {};
    \node[vertex, label=below:{\contour{white}{\small $\left\{ e_{14}\right\}{\zero}$}}] at (F4) {};
    
    \node[vertex, label=below:{\contour{white}{\small $\left\{ v_{1}\right\}{\zero}$}}] at (E1) {};
    \node[vertex, label=right:{\contour{white}{\small $\left\{ v_{2}\right\}{\zero}$}}] at (E2) {};
    \node[vertex, label=above:{\contour{white}{\small $\left\{ v_{3}\right\}{\zero}$}}] at (E3) {};
    \node[vertex, label=right:{\contour{white}{\small $\left\{ v_{4}\right\}{\zero}$}}] at (E4) {};

    \end{scope}
\end{tikzpicture}
    \caption{\label{figure:vertices_St(P)} Vertices of the dual stellahedral transform of a square.}
\end{figure}
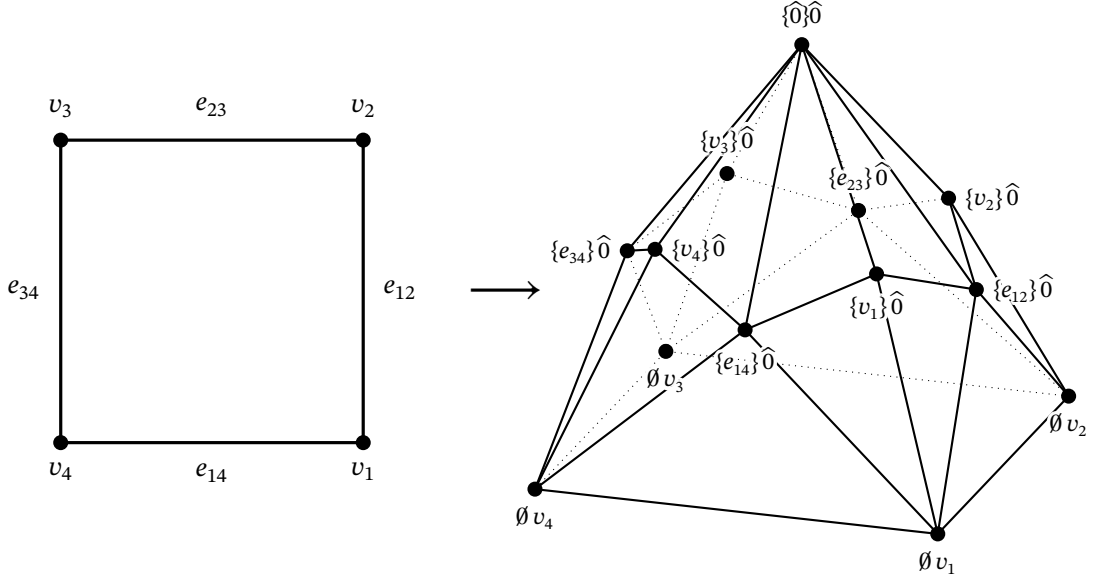
The dual stellahedral transform of a polytope always contains a face that is affinely equivalent to the polytope itself, since we do not subdivide any face of the {base} of the pyramid. In Figure \ref{figure:faces_St(P)}, we may see some labels of faces of $\DStell(\mathcal{P})$. Notice that the faces that do not intersect the base are all simplices. 
\begin{figure}[ht!]
    \centering
    \begin{tikzpicture}
    \tdplotsetmaincoords{70}{108} % Adjust view angle
    \begin{scope}[tdplot_main_coords, 
    scale=3, 
    vertex/.style={circle,fill,inner sep=2pt}]
    \tikzstyle help lines=[color=blue!50, thin]

    % ---- Coordinates ----
    % Base vertices
    \coordinate (B1) at (1.0, 1.0, 0.0);
    \coordinate (B2) at (-1.0, 1.0, 0.0);
    \coordinate (B3) at (-1.0, -1.0, 0.0);
    \coordinate (B4) at (1.0, -1.0, 0.0);

    % Apex
    \coordinate (A)  at (0.0, 0.0, 2.0);

    % Stellar subdivision of the faces (Facets)
    \coordinate (F1) at (0.0, 0.8666666667, 0.8666666667);
    \coordinate (F2) at (-0.8666666667, 0.0, 0.8666666667);
    \coordinate (F3) at (0.0, -0.8666666667, 0.8666666667);
    \coordinate (F4) at (0.8666666667, 0.0, 0.8666666667);

    % Stellar subdivision of the edges
    \coordinate (E1) at (0.55, 0.55, 1.1);
    \coordinate (E2) at (-0.55, 0.55, 1.1);
    \coordinate (E3) at (-0.55, -0.55, 1.1);
    \coordinate (E4) at (0.55, -0.55, 1.1);

    % Red Face Label
    \coordinate (RF) at (1, -1, 0.6);
    \coordinate (RF') at (1, 2, 0.75);
    \coordinate (RF'') at (1.3, -0.5, -0.1);

    % ---- Base Edges ----
    % Base back edges
    \draw[thick] (B4) -- (B1) -- (B2);
    \draw[dotted] (B2) -- (B3) -- (B4);

    % Hidden Face/Edge nodes to Base
    \draw[dotted] (B2) -- (F2);
    \draw[dotted] (B3) -- (F2);
    \draw[dotted] (B3) -- (F3);
    \draw[dotted] (B3) -- (E3);
    % Visible Face/Edge nodes to Base
    \draw[thick] (B1) -- (F1);
    \draw[thick] (B1) -- (F4);
    \draw[thick] (B1) -- (E1);
    \draw[thick] (B2) -- (F1);
    \draw[thick] (B2) -- (E2);    
    \draw[thick] (B4) -- (F3);
    \draw[thick] (B4) -- (F4);
    \draw[thick] (B4) -- (E4);

    % Hidden Face nodes to Edge nodes
    \draw[dotted] (F2) -- (E2);
    \draw[dotted] (F2) -- (E3);
    \draw[dotted] (F3) -- (E3);
    % Visible Face nodes to Edge nodes
    \draw[thick] (F1) -- (E1);
    \draw[thick, color=red!80] (F1) -- (E2);
    \draw[thick] (F3) -- (E4);
    \draw[thick] (F4) -- (E4);
    \draw[thick] (F4) -- (E1);

    % Hidden Edge/Face nodes to Apex
    \draw[dotted] (F2) -- (A);    
    \draw[dotted] (E3) -- (A);    
    % Visible Edge/Face nodes to Apex
    \draw[thick] (E1) -- (A);
    \draw[thick] (F1) -- (A);
    \draw[thick] (E2) -- (A);
    \draw[thick] (F3) -- (A);
    \draw[thick] (E4) -- (A);
    \draw[thick] (F4) -- (A);    

    % Face 0<v_1<e_14,v_1
    \draw[fill=red!50,fill opacity=0.75] (A) -- (E1) -- (F4) -- cycle;
    \draw[fill=red!50,fill opacity=0.75] (B1) -- (B4) -- (F4) --  cycle;

    \node[color=red, label=below:{\color{BrickRed}\contour{white}{ $\{{\zero}<v_1<e_{14}\}\zero$}}] at (RF) {};
    \node[color=red, label=below:{\color{BrickRed}\contour{white}{ $\{v_2<e_{12}\}\zero$}}] at (RF') {};
    \node[color=red, label=below:{\color{BrickRed}\contour{white}{ $\{e_{14}\}e_{14}$}}] at (RF'') {};
    
    % --- Arrow to Barycenter ---

    \coordinate (FaceCenter) at (barycentric cs:A=1,E1=1,F4=1);
    \coordinate (FaceCenter') at (barycentric cs:F1=1,E2=1);
    \coordinate (FaceCenter'') at (barycentric cs:B1=1,B4=1,F4=1);

    \draw[-{Stealth[inset=1pt,scale length=1.5]}, thick, BrickRed] 
        (RF) .. controls +(0.4, 0.25, 0.5) and +(0.3, -0.2, 0.05) .. (FaceCenter);
    \draw[-{Stealth[inset=1pt,scale length=1.5]}, thick, BrickRed] 
        (RF') .. controls +(0.1, 0.25, 0.3) and +(0.1, 0.7, 0.03) .. (FaceCenter');      
    \draw[-{Stealth[inset=1pt,scale length=1.5]}, thick, BrickRed] 
        (RF'') .. controls +(0.4, 0.25, 0.5) and +(0.3, -0.2, 0.05) .. (FaceCenter'');

    % ---- Draw Vertices ----
    \node[vertex, label=below:{\contour{white}{\small $\emptyset v_1$}}] at (B1) {};

    \node[vertex, label=below:{\contour{white}{\small $\emptyset v_4$}}] at (B4) {};
    
    \node[vertex, label=above:{\small $\{{\zero}\}{\zero}$}] at (A) {};
    
    \node[vertex, label=right:{\contour{white}{\small $\{e_{12}\}{\zero}$}}] at (F1) {};
    \node[vertex, label=below:{\contour{white}{\small $\{e_{14}\}{\zero}$}}] at (F4) {};
    
    \node[vertex, label=below:{\contour{white}{\small $\{v_1\}{\zero}$}}] at (E1) {};
    \node[vertex, label=right:{\contour{white}{\small $\{v_{2}\}{\zero}$}}] at (E2) {};
    \end{scope}
\end{tikzpicture}
    \caption{\label{figure:faces_St(P)} Faces of the dual stellahedral transform of a square.}
\end{figure}
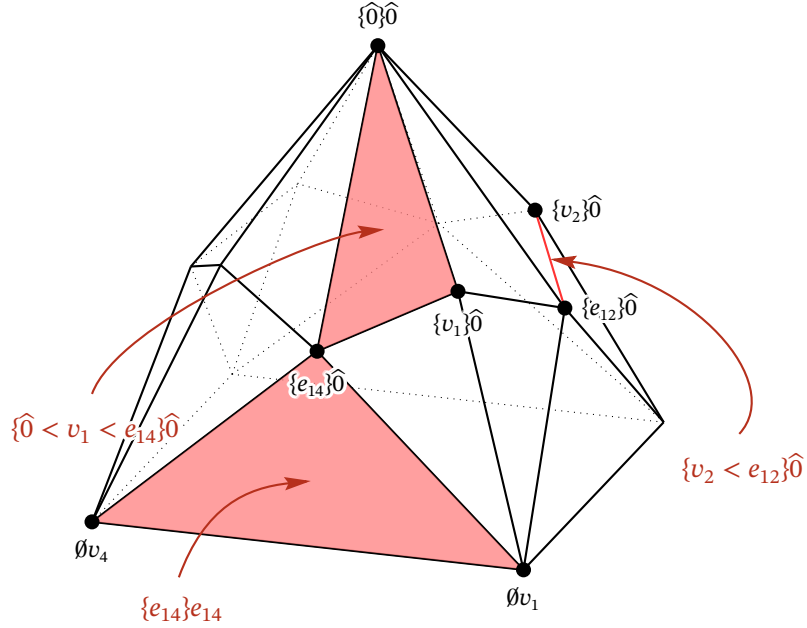
\begin{proof}[Proof of Theorem \ref{thm:face_lattice_St(P)}]
        In order to prove that $\Phi$ is an isomorphism, it is sufficient to show that it is well-defined, bijective, and order-reflecting, since it is order-preserving by construction. 

Let $Q$ be any face of $\DStell(\mathcal{P})$ that is disjoint from the base $\mathcal{P}$. If one performs stellar subdivisions on all nonempty proper faces of the pyramid $\operatorname{Pyr}(\mathcal{P})$ in decreasing order of dimension, one would obtain a polytope whose boundary complex is barycentric subdivision of the boundary complex of $ \operatorname{Pyr}(\mathcal{P})$. In the construction of $\DStell(\mathcal{P})$ we only skip subdivisions of faces contained in the base, hence the poset structure of the faces disjoint from $\mathcal{P}$ does not change. Consequently, every face of $\DStell(\mathcal{P})$ disjoint from the base $\mathcal{P}$ is a simplex of the form
\[
v_{F_1} \star \dots \star v_{F_m}
\]
where $F_1 \subsetneq \dots \subsetneq F_m \neq {\one}$ is a chain of faces of $\mathcal{P}$. It follows that $\Phi$ is well-defined on the pairs $(\mathcal{C}, F) \in S$ such that $F = \zero$, which are in bijection with the faces that are disjoint from the base since $\Phi$ restricts to the function
\begin{align*}
      \left(\left\{F_1 \subsetneq \dots \subsetneq F_m\right\}, \zero \right) &\longmapsto v_{F_1} \star \dots \star v_{F_m}.
\end{align*}
Now, consider any proper face $Q$ of $\DStell(\mathcal{P})$ that does intersect the base $\mathcal{P}$. 
Every proper face containing both the apex and a proper face of $\mathcal{P}$ has been subdivided during the construction, so $Q$ cannot contain the apex $v_{\oldhat{0}}$. Therefore, $Q$ is the convex hull of a face $F \in \mathcal{F}(\mathcal{P})$ and a (possibly empty) set of vertices $v_{G_1},\dots,v_{G_m}$, for some non-empty proper faces  $G_1,\dots,G_m \in \mathcal{F}(\mathcal{P})$. 
Notice that all the faces of $\operatorname{Pyr}(\mathcal{P})$ not contained in the base have been removed and replaced with faces of the form $v_G \star  H $ where $G$ is a proper face of $\mathcal{P}$, and $H$ is a proper face of a face $H'$ containing $v_{\oldhat{0}}\star G$ but that does not itself contain $ v_{\oldhat{0}} \star  G$. 
We observe that $H'$ must satisfy $H' \cap \mathcal{P} = (v_{\oldhat{0}} \star  G) \cap \mathcal{P} = G$. Otherwise, during the subdivision step of $v_{\oldhat{0}} \star  G$, the intermediate polytope of the construction would contain as a subset the face $v_{\oldhat{0}} \star  G'$ for some $G \subsetneq G'$, which is impossible since the order of the subdivisions is by decreasing dimension. For this reason, we must have $F \subseteq G_i$ for every $i\in [m]$. Assume without loss of generality that $v_{G_1},\dots,v_{G_m}$ are ordered such that $i < j$ if and only if $v_{G_i}$ has been introduced later than $v_{G_j}$ in a subdivision step. We prove by induction on $m$ that $G_{1} \subsetneq \dots \subsetneq G_m$. If $m=1$, the statement is trivially true. If $m>1$, we notice that the face $Q$ must be $v_{G_1} \star H$ where $H$ is a face of $\DStell(\mathcal{P})$ and in particular it is $H=v_{G_2} \star \dots \star v_{G_m} \star F$ such that $G_2 \subsetneq \dots \subsetneq G_m$ by the induction hypothesis. This implies that there exists an edge of $\DStell(\mathcal{P})$ between $v_{G_1}$ and $v_{G_2}$, from which it follows that $G_1 \subsetneq G_2 \subsetneq \dots \subsetneq G_m$ by the previous discussion on faces that do not intersect the base. This implies that $Q = \Phi \left(\left(\left\{G_1 \subsetneq \dots \subsetneq G_m \right\},F\right)\right)$. 
Furthermore, we have that for any chain of faces $\emptyset \neq F \subseteq G_1 \subsetneq G_2 \subsetneq \dots \subsetneq G_m \subsetneq \mathcal{P} $ in $\mathcal{F}(\mathcal{P})$ with $m \geq 1$,  the set $Q= \operatorname{conv}(F, v_{G_1},\dots,v_{G_m})$ is a face of $\DStell(\mathcal{P})$. Indeed, let 
\[
Q_j := (v_{\oldhat0} \star G_j) \star v_{G_{j+1}} \star \cdots \star v_{G_m}
\]
for any $j=1,\dots,m$. For $j=m$, the set $Q_m = v_{\oldhat0} \star G_m$ is a face of $\operatorname{Pyr}(\mathcal{P})$, before having applied any stellar subdivision. For $j<m$, we see by induction that $Q_j$ is a face of the intermediate polytope obtained during the construction of $\DStell(\mathcal{P})$, just after applying a stellar subdivision to the face $\mathcal{v}_{\oldhat{0}}\star G_{j+1}$. The set $Q$ is obtained as a face of $\DStell(\mathcal{P})$ from $Q_1=  (v_{\oldhat0} \star G_1) \star v_{G_{2}} \star \cdots \star v_{G_m}$ after a stellar subdivision of the face $\mathcal{v}_{\oldhat{0}}\star G_1$, and remains unchanged during the rest of the construction of $\DStell(\mathcal{P})$.
Hence, we have a bijection between the faces that intersect the base and the pairs $(\mathcal{C}, F)$ of $S$ such that $F \neq \emptyset$. 
Ultimately, we have that $\Phi$ is a bijection such that
  \begin{align*}
        \Phi: S \sqcup \{\one_S\} &\to  \mathcal{F}(\DStell(\mathcal{P}))   \\
        \{(\mathcal{C}, F) \in S \mid F=\emptyset\} &\leftrightarrow \{F \in \mathcal{F}(\DStell(\mathcal{P}))  \mid F \cap \mathcal{P}=\emptyset\} \\
        \{(\mathcal{C}, F) \in S \mid F\neq\emptyset\} &\leftrightarrow \{F \in \mathcal{F}(\DStell(\mathcal{P}))  \mid F \cap \mathcal{P}\neq\emptyset\} \setminus\{\DStell(\mathcal{P})\} \\
        \one_S &\leftrightarrow \one_{\DStell(\mathcal{P})},
    \end{align*}
Face inclusion is order-reflecting by vertex labels, which concludes the proof.
\end{proof}
The previous theorem proves that the stellahedral transform preserves the property of a poset of being the face poset of a polytope. In particular, we have the following corollary.
\begin{corollary}\label{cor:stellahedral_poset_polytope}
    Let $\mathcal{P}$ be a polytope. Then $\mathcal{F}(\DStell(\mathcal{P})) \cong \Stell(\mathcal{F}(\mathcal{P}^{\ast}))$.
\end{corollary}
For any $n$-dimensional polytope $\mathcal{P}$ we define its \textit{$f$-polynomial} $f_{\mathcal{P}}(x)$ as
\[
f_{\mathcal{P}}(x) := \sum_{i=0}^n |\{F \in \mathcal{F}(\mathcal{P})\mid \dim(F)=i\}|\cdot x^i.
\]
It is possible to calculate the $f$-polynomial of the dual stellahedral transform with the following result. 
\begin{proposition}
    Let $\mathcal{P} $ be an $n$-dimensional polytope. Then the $f$-polynomial of $\DStell(\mathcal{P})$ is
    \[
    f_{\DStell(\mathcal{P})}(x) = x^{n+1} + x^n + 1 + \sum_{\substack{I \subseteq \{0,\dots,n-1\}\\I \neq \emptyset} } f_{\mathcal{P}}(I) x^{|I|-1}(x+1)(x^{\min(I)}+1)
    \]
    where $f_{\mathcal{P}}$ denotes the flag $f$-vector of the polytope $\mathcal{P}$.
\end{proposition}
\begin{proof}
Up to an isomorphism, we may work with the poset $S\sqcup\{\one_S\}$ as defined in Theorem \ref{thm:face_lattice_St(P)}. 
We count the nonempty faces of $\DStell(\mathcal{P})$ by taking non-empty chains $\{C_1\subsetneq   \dots\subsetneq C_k \}$ of nonempty proper faces of $\mathcal{P}$ and associating to each of them four different faces of $\DStell(\mathcal{P})$:
\begin{itemize}
    \item $(\{{\zero} \subsetneq C_1 \subsetneq...\subsetneq C_k\}, {\zero})$, which contributes $x^{k}$ to the $f$-polynomial of $\DStell(\mathcal{P})$;
    \item $(\{C_1 \subsetneq...\subsetneq C_k\}, {\zero})$, which contributes $x^{k-1}$ to the $f$-polynomial of $\DStell(\mathcal{P})$;
    \item $\left(\{C_1 \subsetneq...\subsetneq C_k\}, C_1\right)$, which contributes $x^{\dim(C_1)+k}$ to the $f$-polynomial of $\DStell(\mathcal{P})$;
    \item $\left(\{ C_2 \subsetneq...\subsetneq C_k\}, C_1\right)$, which contributes $x^{\dim(C_1)+k-1}$ to the $f$-polynomial of $\DStell(\mathcal{P})$.
    \end{itemize}
    In this way, we count exactly once all the nonempty faces of $\DStell(\mathcal{P})$, except for the whole polytope $\one_{\DStell(\mathcal{P})}$, the base $(\emptyset,\mathcal{P})$, and the apex $(\{{\zero}\},{\zero})$, which have been missed and contribute respectively $x^{n+1}$, $x^n$ and $1$ to the $f$-polynomial of $\DStell(\mathcal{P})$. Hence, by summing over all possible chains we get
    \begin{align*}
       f_{\DStell(\mathcal{P})}(x) &= x^{n+1} + x^n + 1 +\sum_{c \in \Delta\big(\overline{\mathcal{F}(\mathcal{P})}\big)\setminus\big\{\emptyset\big\}}\left( x^{|c|} +  x^{|c|-1}+ x^{|c|+\dim(\min c)} + x^{|c|-1+\dim(\min c)}\right)\\
       &= x^{n+1} + x^n + 1 +\sum_{c \in \Delta\big(\overline{\mathcal{F}(\mathcal{P})}\big)\setminus\big\{\emptyset\big\}} x^{|c|-1}(x+1)(x^{\operatorname{dim}(\operatorname{min}c)}+1) \\
        &=   x^{n+1} + x^n + 1 + 
    \sum_{\substack{I \subseteq \{0,\dots,n-1\}\\I \neq \emptyset} } f_{\mathcal{P}}(I) x^{|I|-1}(x+1)(x^{\min(I)}+1),
    \end{align*}
    where the last equality is obtained by grouping the chains with the same sequence of dimensions of their elements.
\end{proof}

\section{Stellahedral transform and Chow theory}\label{sec:chow-theory}

Throughout this section we assume that $P$ is an Eulerian poset. All the KLS functions, Chow functions and augmented Chow functions that we will consider are those associated to the Eulerian kernel $\varepsilon$. 

\subsection{Augmented Chow polynomials}

The following key result characterizes particular values of the left KLS function of $\Stell(P)$ via the right KLS functions of $P$.  
For the remainder of this work, let $\tilde{\rho}$ denote the rank function of $\Stell(P)$, and let $\tilde{\varepsilon}$ be the Eulerian kernel in its incidence algebra. Correspondingly, we will use $\tilde{g}$ and $\tilde{h}$ to represent the toric $g$- and $h$-functions of $\Stell(P)$.

\begin{proposition}
For any $(c,y) \le (c',y')<{\one}_{\Stell(P)}$ in $\Stell(P)$, we have the following formula: $$\tilde{g}_{(c,y), (c',y')}(x)=f_{y'y}(x),$$ 
where $f$ is the right KLS function of $P$.
\label{prop: g of st}
\end{proposition}
\begin{proof}
We start by defining the function:
$$ z_{(c,y), (c',y')}(x) = 
\begin{cases} 
f_{y'y}(x) & \text{if } (c',y')\neq {\one}_{\Stell(P)}
\\ 
\tilde{g}_{(c,y), {\one}_{\Stell(P)}}(x) & \text{otherwise} 
\end{cases} $$
 KLS functions are determined interval by interval. By Theorem~\ref{thm:kls_functions} and the uniqueness of KLS functions, in order to show the validity of our formula, it suffices to show that $z$ satisfies the following three properties on intervals ending below the top:
    \begin{enumerate}[(i')]
        \item\label{it:first} $z_{(c,y),(c,y)} = 1$,
        \item\label{it:second} $\deg{z_{(c,y),(c',y')}(x)} < \frac{1}{2} \tilde{\rho}_{(c,y),(c',y')}$ for any $(c,y) < (c',y')$, and
        \item\label{it:third} $z^{\rev} = z\cdot\tilde{\varepsilon}$ whenever $(c',y')<{\one}_{\Stell(P)}$.
    \end{enumerate}
Note that \ref{it:first} is immediate from the fact that $f$ is the KLS function of $P$, hence $z_{(c,y), (c,y)}(x) = f_{yy}(x) = 1$.

Similarly, \ref{it:second} follows again from the fact that $f$ is the KLS function of $P$. Suppose $(c,y) < (c',y')$. We have $\deg(z) = \deg(f_{y'y}(x))$. If $y' < y$ we have 
$$\deg(f_{y'y}(x)) < \frac{1}{2}\rho_{y'y} \leq \frac{1}{2}\tilde{\rho}_{(c,y), (c',y')}. $$
On the other hand if $y' = y$ we must have $c \subsetneq c'$, so $|c'| - |c| \ge 1$ and we have
$$
\deg(f_{yy}(x)) = 0 < \frac{|c'|-|c|}{2} = \frac{1}{2}\tilde{\rho}_{(c,y), (c',y')}.
$$
Now we prove \ref{it:third}:
\begin{align*}
(z \cdot \tilde{\varepsilon})_{(c,y), (c',y')} &= 
\sum_{(d,w) \in [(c,y), (c',y')]} z_{(c,y), (d,w)}(x) \cdot \tilde{\varepsilon}_{(d,w), (c',y')}(x) \\
&= \sum_{(d,w) \in [(c,y), (c',y')]} f_{wy}(x) (x-1)^{\rho_{y'w} + |c'| - |d|} \\
&= \sum_{w \in [y', y]} \sum_{d \in [c, c']} f_{wy}(x) (x-1)^{\rho_{y'w} + |c'| - |d|} \\
&= \left( \sum_{w \in [y', y]} f_{wy}(x) (x-1)^{\rho_{y'w}} \right) \left( \sum_{d \in [c, c']} (x-1)^{|c'| - |d|} \right).
\end{align*}
The first factor is the convolution $( \varepsilon\cdot f)_{y'y}$, which equals $f_{y'y}^{\rev}(x)$ since $f$ is the right KLS function of $P$. The second factor is the convolution of $\zeta$ and the Eulerian kernel in the boolean poset $[c, c']$, which equals $\zeta^{\rev}$ in the boolean poset, evaluating to $x^{|c'|-|c|}$. 
Thus,
\[
(\varepsilon\cdot f)_{y'y}x^{|c'|-|c|} = x^{\rho_{y'y}+|c'|-|c|} f_{y'y}(x^{-1}) = z^{\rev}_{(c,y), (c',y')}. \qedhere
\]
\end{proof}

Now, having the above formula for the left KLS function of $\Stell(P)$, we can prove Theorem~\ref{thm:augChow_is_toric-h_of_Stell}.

\begin{theorem}
Let $P$ be an Eulerian poset. The right augmented Chow polynomial of $P$ equals the toric $h$-polynomial of $\Stell(P)$:
$$
F_{P}(x)=h_{\Stell(P)}(x).
$$
\end{theorem}
\begin{proof}
Let $n$ be the rank of $P$. By the definition of the toric $h$-polynomial, we have
\[
h_{\Stell(P)}(x) = \frac{ \tilde{g}^{\,\rev}_{(\{{\zero}\}, {\one}), {\one}_{\Stell(P)}} - \tilde{g}_{(\{{\zero}\}, {\one}), {\one}_{\Stell(P)}}}{x-1}. 
\]
Since $\tilde{g}^{\,\rev} = \tilde{g} \cdot \tilde{\varepsilon}$, we can write
\[
h_{\Stell(P)}(x) = \frac{1}{x-1} \sum_{(\{{\zero}\}, {\one}) \leq (d,y) < {\one}_{\Stell(P)}} \tilde{g}_{(\{{\zero}\}, {\one}), (d,y)}(x) (x-1)^{\tilde{\rho}_{(d,y), {\one}_{\Stell(P)}}}.
\]
From Lemma \ref{lm: rank}, we know that the co-rank of an arbitrary element $(d,y)$ is 
\[
\tilde{\rho}_{(d,y), {\one}_{\Stell(P)}} = \rho_{{\zero},y}-|d|+2,
\]
hence using Proposition \ref{prop: g of st} we have that the toric $h$-polynomial of $\Stell(P)$ is
\[
\sum_{(\{{\zero}\}, {\one}) \leq (d,y) < {\one}_{\Stell(P)}} f_{y{\one}}(x) (x-1)^{\rho_{{\zero}y} - |d|+1} .
\]
We can split the sum by isolating the case $y = {\zero}$. If $y = {\zero}$, the only valid chain is $d = \{{\zero}\}$, so $|d| = 1$ and $\rho_{{\zero}{\zero}} = 0$. The corresponding term evaluates to $f_{{\zero}{\one}}(x)$. For $y > {\zero}$, we split the sum over chains $d$ bounded by $y$ and containing ${\zero}$ into two sets: those that do not contain $y$ and those that contain $y$. Let $c = d \setminus \{{\zero}, y\}$, so $c$ is a face in the order complex $\Delta({\zero},y)$.
If $y \notin d$, then $|d| = |c| + 1$; if $y \in d$, then $|d| = |c| + 2$. Thus, we have
\begin{align*}
h_{\Stell(P)}(x) &= f_{{\zero}{\one}}(x) + \sum_{y > {\zero}} f_{y{\one}}(x) \sum_{c \in \Delta({\zero},y)} \left[ (x-1)^{\rho_{{\zero}y} - |c|} + (x-1)^{\rho_{{\zero}y} - |c| - 1} \right] \\
&= f_{{\zero}{\one}}(x) + \sum_{y > {\zero}} f_{y{\one}}(x) \sum_{c \in \Delta({\zero},y)} x(x-1)^{\rho_{{\zero}y} - |c| - 1}.
\end{align*}

Using the definition of the $h$-polynomial of simplicial complexes (where $\dim \Delta({\zero},y) = \rho_{{\zero}y}-2$) we can write:
\[
h_{\Stell(P)}(x) = f_{{\zero}{\one}}(x) + \sum_{y > {\zero}} f_{y{\one}}(x) x h_{\Delta({\zero},y)}(x) =f_{{\zero}{\one}}(x) + \sum_{y > {\zero}} f_{y{\one}}(x) x \H_{{\zero},y}(x).
\]
Now we use that $x \H_{s,t}(x)=\H^{\rev}_{s,t}(x)$ for every $s<t$, as proved during the proof of Proposition 3.6 of \cite{ferroni-matherne-vecchi}
\[
h_{\Stell(P)}(x) = f_{{\zero}{\one}}(x) + \sum_{y > {\zero}} f_{y{\one}}(x) \H^{\rev}_{{\zero},y}(x) =  \sum_{y\in P} f_{y{\one}} \H_{{\zero},y}^{\operatorname{rev}},
\]
and then we conclude using \cite[Prop. 3.6 and Prop. 3.15]{ferroni-matherne-vecchi}
\[
h_{\Stell(P)}(x) = (\H^{\operatorname{rev}}\cdot f)_P=(\H \cdot \varepsilon \cdot f)_P = (\H \cdot f^{\operatorname{rev}})_P = F_P.\qedhere
\]
\end{proof}

\subsection{Eulerian Chow polynomial of the stellahedral transform}
This subsection is not essential for the rest of the article, but we consider it relevant in the study of Chow polynomials of Eulerian posets. The following can be viewed as a counterpart of \cite[Section~4]{ferroni-riccardi}.

According to \cite[Theorem~5.4]{ferroni-matherne-vecchi}, the Eulerian Chow polynomial of an Eulerian poset $P$ coincides with the $h$-polynomial of its proper order complex, $\Delta(\bar{P})$. As we demonstrate below, the $f$-vector of $\Delta(\bar{P})$ completely determines the $f$-vector of $\Delta\left(\overline{\Stell(P)}\right)$. Consequently, the Eulerian Chow polynomial of $P$ uniquely determines the Eulerian Chow polynomial of $\Stell(P)$.
\begin{theorem}
Let $P$ be an Eulerian poset of rank $r$. The Eulerian Chow polynomial $\H_{\Stell(P)}(x)$ is determined by the Eulerian Chow polynomial $\H_P(x)$ via the following formulas
\begin{align*}
\frac{\H_{\Stell(P)}(x)}{(1-x)^{r+1}} &= \sum_{k=0}^{\infty} N(k)x^k\\
&= \sum_{k=0}^\infty \zeta_P(k+1) P_k(x)
\end{align*}
where $P_k(x)$ is the polynomial defined by $P_k(x) := \sum_{m=0}^k x^m \frac{m^{k-m}}{(m+1)^{k-m+1}}$ and $N(k)$ is the value $N(k) := \left[\left((k+1)\delta-k\zeta)\right)^{-1}\cdot\zeta^{k+1}\right]_{{\zero}{\one}}$.
\end{theorem}

\begin{proof}
Recall that the zeta polynomial of a poset counts its multichains by size (see \cite[Section~3.12]{stanley-ec1}). We denote the zeta polynomial of a poset $Q$ by $\zeta_Q(n)$. Specifically, $\zeta_Q(n)$ counts the number of multichains $\{q_1 \leq \dots \leq q_{n-1}\}$ in $Q$ of size $n-1$.
From Theorem 5.4 of \cite{ferroni-matherne-vecchi}, we know that the Eulerian Chow polynomial of $\Stell(P)$ satisfies the identity
$$ \sum_{n=0}^\infty \zeta_{\Stell(P)}(n+1)x^n = \frac{\H_{\Stell(P)}(x)}{(1-x)^{r+2}} $$
since the rank of $\Stell(P)$ is $r+1$. Now we proceed to compute $\zeta_{\Stell(P)}(n+1)$. 
By definition, $\zeta_{\Stell(P)}(n+1)$ counts the number of multichains $\{z_1 \le z_2 \le \dots \le z_{n}\}$ of size $n$ in $\Stell(P)$. 
Since the maximal element is ${\one}_{\Stell(P)}$, any such multichain consists of $m$ elements strictly below ${\one}_{\Stell(P)}$ followed by $n-m$ copies of ${\one}_{\Stell(P)}$ for some $m \leq n$. The first $m$ elements in $ \Stell(P) \setminus \{{\one}_{\Stell(P)}\}$ form a multichain $(c^{(1)}, y^{(1)}) \le \dots \le (c^{(m)}, y^{(m)})$.
By the order relations in $\Stell(P)$, this implies that $\{y^{(m)} \le \dots \le y^{(1)}\}$ is a multichain in $P$ and $c^{(1)} \subseteq \dots \subseteq c^{(m)}$ are chains in $P$.

Assume first that $m>0$. Fix $c = c^{(m)}$ any strict chain $\{{\zero} = x_0 < x_1 < \dots < x_k \le y^{(m)}\}$. 
The condition $c^{(1)} \subseteq \dots \subseteq c^{(m)}$ can be seen as a choice of the birth time of each element of $c$. 
Hence, there are exactly $m^k$ such choices. 
Thus, the number of multichains in $ \Stell(P) \setminus \{{\one}_{\Stell}\}$ of size $m$ is given by
$$ N(m) = \sum_{c} m^{|c|-1} \sum_{y_m \ge \max c, \, y_m \le \dots \le y_1} 1. $$
Notice that the inner sum counts the number of multichains $\{y \le y_m \le \dots \le y_1 \le {\one}\}$ where $y = \max c$, which is exactly counted by $\zeta^{m+1}_{y {\one}}$ where $\zeta$ is the zeta function of $P$.

For the outer sum, fix an element $y$ to be the maximum of $c$. The number of strict chains of length $k$ between ${\zero}$ and $y$ is given by the strict incidence function $\eta^k_{{\zero}y}$, where $\eta=\zeta - \delta$. Hence, we have
\begin{align*}
    N(m) &= \sum_{y\in P} \left(\sum_{k=0}^{\infty} m^k \eta^k_{{\zero}y}\right)\zeta^{m+1}_{y{\one}} \\
    &=  \sum_{y\in P} (\delta-m\eta)^{-1}_{{\zero}y}\zeta^{m+1}_{y{\one}} \\
    &=\left[(\delta-m\eta)^{-1}\cdot\zeta^{m+1}\right]_{{\zero}{\one}},
\end{align*}
where in the second equality we used the algebraic identity $(\delta-f)^{-1}=\sum_{k=0}^{\infty}f^k$. The formula also holds for $m=0$.
Writing $\eta=\zeta-\delta$ and using the same algebraic identity, we obtain
$$
N(m)=\frac{1}{m+1}\left[ \left(\delta-\frac{m}{m+1}\zeta\right)^{-1}\cdot \zeta^{m+1}\right]_{{\zero}{\one}}= \sum_{j=0}^\infty\frac{m^j}{(m+1)^{j+1}}\zeta_P(j+m+1).
$$
Since we have $\zeta_{\Stell(P)}(n+1) = \sum_{m=0}^n N(m)$, we may substitute it into the formula of the Eulerian Chow polynomial of $\Stell(P)$:
\begin{align*}
    \frac{\H_{\Stell(P)}(x)}{(1-x)^{r+2}} &= \sum_{n=0}^\infty \zeta_{\Stell(P)}(n+1) x^n = \sum_{n=0}^\infty \sum_{m=0}^nN(m) x^n \\&=\sum_{m=0}^\infty N(m)x^m\sum_{n=m}^\infty x^{n-m} = \frac{1}{1-x}\sum_{m=0}^\infty N(m)x^m\\
    &= \frac{1}{1-x}\sum_{m=0}^\infty x^m \sum_{j=0}^\infty \frac{m^j}{(m+1)^{j+1}} \zeta_P(j+m+1). 
\end{align*}
Notice that by the second equality, we get the formula
\[
\frac{\H_{\Stell(P)}(x)}{(1-x)^{r+1}} = \sum_{k=0}^{\infty} N(k)x^k.
\]
Notice that in the last equality, if we let $k = j+m$, the inner sum ranges from $k=m$ to $\infty$. If we swap the order of summation, summing first over $k$ from $0$ to $\infty$ and then over $m$ from $0$ to $k$, we isolate $\zeta_P(k+1)$:
\[ \frac{\H_{\Stell(P)}(x)}{(1-x)^{r+1}} = \sum_{k=0}^\infty \zeta_P(k+1) \left( \sum_{m=0}^k x^m \frac{m^{k-m}}{(m+1)^{k-m+1}} \right) = \sum_{k=0}^\infty \zeta_P(k+1) P_k(x). \qedhere\]
\end{proof}

\section{Unimodality and \texorpdfstring{${\gamma}$}{gamma}-positivity}\label{sec:unimodality-gamma}

In this section, we discuss the unimodality and $\gamma$-positivity of augmented Chow polynomials. We first provide a formula for the right augmented Chow polynomial in terms of the $cd$-index. We then prove that the augmented Chow polynomials of Gorenstein* posets are unimodal and that those of face posets of polytopes are $\gamma$-positive. Throughout the section, the augmented Chow, Chow and KLS polynomials are those associated to the Eulerian kernel.

\subsection{Augmented Chow polynomials and the \texorpdfstring{${\mathbf{cd}}$}{cd}-index} \label{subsec:augmented-cd}

Every Eulerian poset $P$ of rank $n+1$ can be associated with a bivariate homogeneous polynomial $\Phi_P(\mathbf{c},\mathbf{d})\in \mathbb{Z}\langle \mathbf{c}, \mathbf{d} \rangle$ of degree $n$ in two non-commuting variables $\mathbf{c}$ and $\mathbf{d}$, often called the \emph{cd-index} of $P$, where $\deg\mathbf{c}=1$ and $\deg\mathbf{d}=2$. We refer to \cite{bayer-cd} for a detailed account of the $cd$-index of Eulerian posets. 

A result of Bayer and Ehrenborg \cite[Theorem~4.2]{bayer-ehrenborg} allows for the computation of the toric $g$-polynomial of an Eulerian poset $P$ from a linear map $g:\Z \langle \mathbf{c} , \mathbf{d} \rangle \to \Z[x]$ such that $g(\Phi_P,x)$ is the toric $g$-polynomial of $P$. 

The $cd$-index of the dual poset $P^{\ast}$ may also be calculated as follows. 
Take the involution given by
\[
 w=u_1u_2\cdots u_m\longmapsto w^{\ast}=u_m\cdots u_2u_1,
 \qquad u_j\in\{\mathbf{c},\mathbf{d}\},
\]
and extend it linearly to $\Z\langle\mathbf{c},\mathbf{d}\rangle$.
Then 
\[
\Phi_{P^{\ast}} = (\Phi_P)^{\ast}.
\]
Since the right KLS polynomial of $P$ is the toric $g$-polynomial of $P^{\ast}$, it may be calculated from a linear map $f:\Z \langle \mathbf{c} , \mathbf{d} \rangle \to \Z[x]$ such that $f(\Phi_P,x) = g(\Phi_{P^{\ast}},x)$ by setting 
\[
f(w,x) := g(w^{\ast},x)
\]
on every $cd$-polynomial $w$.

The Chow polynomial of $P$ can also be calculated from the $cd$-index. In particular, Ferroni, Matherne and Vecchi have proved that it is equal to the $h$-polynomial of the proper order complex $\Delta(\bar{P})$ \cite[Theorem~5.4]{ferroni-matherne-vecchi}, which Gal \cite[p.~275]{gal} observes to be $\Phi_P(1+x,2x)$. Therefore, there also exists a linear function $\H:\Z \langle \mathbf{c} , \mathbf{d} \rangle \to \Z[x]$ defined by 
\[
\H(w(\mathbf{c},\mathbf{d}),x) = w(1+x,2x)
\] 
for every $cd$-polynomial $w$, such that $\H(\Phi_P,x)$ is the Chow polynomial of $P$. \smallskip

Let $\Delta$ be the Newtonian coproduct on $\mathbb{Z} \langle\mathbf{c},\mathbf{d}\rangle$ \cite{ehrenborg-readdy-coproduct},  defined by
\[
 \Delta(1)=0,\qquad \Delta(\mathbf{c})=2(1\otimes1),\qquad
 \Delta(\mathbf{d})=1\otimes\mathbf{c}+\mathbf{c}\otimes1,\qquad
 \Delta(uv)=\Delta(u)(1\otimes v)+(u\otimes1)\Delta(v).
\]
We use the \textit{Sweedler notation}: since $\Delta(w)$ belongs to $\mathbb{Z} \langle\mathbf{c},\mathbf{d}\rangle \otimes \mathbb{Z} \langle\mathbf{c},\mathbf{d}\rangle$, it is a finite sum of pure tensors and can be written as
\[
\Delta(w) = \sum_{w} w_{(1)}\otimes w_{(2)},
\]
where $w_{(1)}$ and $w_{(2)}$ are not specific elements, but the first and second tensor factors across the entire sum. Ehrenborg and Readdy have proved \cite{ehrenborg-readdy-coproduct} that the $cd$-index of an Eulerian poset $P$ satisfies
\begin{equation} \label{eq:cd-index_is_coalgebra_morphism}
\Delta(\Phi_P) = \sum_{\zero < z < \one} \Phi_{[\zero,z]} \otimes \Phi_{[z,\one]}.
\end{equation}

Using the observations above, we may now calculate the right augmented Chow polynomial from the $cd$-index.
\begin{theorem}
    Let $P$ be an Eulerian poset and $F:\Z \langle \mathbf{c} , \mathbf{d} \rangle \to \Z[x]$ be the linear function given by
    \[
     F(w,x):=f(w,x)+x\H(w,x)  +x\sum_w \H({w_{(1)}},x)f({w_{(2)}},x),
    \]
    for any $cd$-polynomial $w$. Then
    \[
        F(\Phi_P,x) =   F_P(x).
    \]
\end{theorem}
\begin{proof}
    Recall that $F=\H\cdot f^{\operatorname{rev}}
 =\H^{\operatorname{rev}}\!\cdot f$ and $\H^{\rev}\! =  x\H + (1-x)\delta $; thus, we have
 \begin{align*}
     F(\Phi_P,x) &= f(\Phi_P,x)+x\H(\Phi_P,x)  +x\sum_{\Phi_P} \H({(\Phi_P)_{(1)}},x)f({(\Phi_P)_{(2)}},x) \\
     &= f_P(x) + x\H_P(x) + \sum_{\zero < z < \one} x\H_{[\zero,z]}(x)f_{[z,\one]}(x)\\
     &= (\H^{\operatorname{rev}}\!\cdot f)_P =F_P(x),
 \end{align*}
 where the second equality follows from equation (\ref{eq:cd-index_is_coalgebra_morphism}).\end{proof}
Notice that for the empty word $1$, we have $F(1,x) = 1+x$. 

For a homogeneous $cd$-polynomial $w$ of degree $n$, write $f(w,x)=\sum_{i\geq0}f_i(w)x^i$. We have $\deg f(w,x)\leq\lfloor n/2\rfloor$; in particular, $f_{\lfloor(n+1)/2\rfloor}(w)=0$ when $n$ is odd. In the following lemma we state a recursive formula for $F$.
\begin{lemma}\label{lem:cd-component-recursions}
Let $w$ be a ${cd}$-monomial of degree $n$. Then
\begin{align*}
F({\mathbf{c}w},x) &= (1+x)F(w,x) + f_{\lfloor (n+1)/2 \rfloor}(w) x^{ \lfloor (n+3)/2 \rfloor},\\
F({\mathbf{d}w},x) &= 2xF(w,x) + f_{\lfloor (n+1)/2 \rfloor}(w)x^{ \lfloor (n+3)/2 \rfloor}(1+x).
\end{align*}
\end{lemma}

\begin{proof}
By Propositions~7.10 and 7.11 of \cite{bayer-ehrenborg}, we know that
\begin{align}
f(\mathbf{c}w,x) &= (1-x) f(w,x) +  f_{\lfloor (n+1)/2 \rfloor}(w) x^{ \lfloor (n+3)/2 \rfloor}, \label{eq:f(cv)}\\
f(\mathbf{d}w,x) &=  f_{\lfloor (n+1)/2 \rfloor}(w) x^{ \lfloor (n+3)/2 \rfloor}. \label{eq:f(dv)}
\end{align}
Therefore, we may evaluate 
\begin{align*}
     F(\mathbf{c}w,x)&=f(\mathbf{c}w,x)+x\H(\mathbf{c}w,x) + 2xf(w,x) +x\sum_w \H(\mathbf{c}{w_{(1)}},x)f({w_{(2)}},x) \\
     &=(1-x) f(w,x) +  f_{\lfloor (n+1)/2 \rfloor}(w) x^{ \lfloor (n+3)/2 \rfloor}\\
     &\quad +x(1+x)\H(w,x) +2 xf(w,x) +x(1+x)\sum_w \H({w_{(1)}},x)f({w_{(2)}},x).
\end{align*}
Factoring out the term $(1+x)$, we obtain
\begin{align*}
     F(\mathbf{c}w,x)&= (1+x) \left( f(w,x) + x\H(w,x) + x \sum_w \H({w_{(1)}},x)f({w_{(2)}},x) \right) \\
     &\quad + f_{\lfloor (n+1)/2 \rfloor}(w) x^{ \lfloor (n+3)/2 \rfloor} \\
     &= (1+x)F(w,x) +  f_{\lfloor (n+1)/2 \rfloor}(w) x^{ \lfloor (n+3)/2 \rfloor}.
\end{align*} 
Similarly,
\begin{align*}
     F(\mathbf{d}w,x)&=f(\mathbf{d}w,x)+x\H(\mathbf{d}w,x) + x\H(\mathbf{c},x)f(w,x) + xf(\mathbf{c}w,x) +x\sum_w \H(\mathbf{d}{w_{(1)}},x)f({w_{(2)}},x) \\
     &=  f_{\lfloor (n+1)/2 \rfloor}(w) x^{ \lfloor (n+3)/2 \rfloor} +2x^2\H(w,x) +x(1+x)f(w,x) \\
     &\quad   +x(1-x) f(w,x) +  f_{\lfloor (n+1)/2 \rfloor}(w) x^{ \lfloor (n+5)/2 \rfloor} +2x^2\sum_w \H({w_{(1)}},x)f({w_{(2)}},x) \\
     &= f_{\lfloor (n+1)/2 \rfloor}(w) x^{ \lfloor (n+3)/2 \rfloor} +2x^2\H(w,x)  +2xf(w,x)\\
     &\quad  +  f_{\lfloor (n+1)/2 \rfloor}(w) x^{ \lfloor (n+5)/2 \rfloor}  +2x^2\sum_w \H({w_{(1)}},x)f({w_{(2)}},x),
\end{align*}
and by factoring out $2x$ we get
\begin{align*}
     F(\mathbf{d}w,x) &= 2x \left( f(w,x) + x\H(w,x)+x\sum_w \H({w_{(1)}},x)f({w_{(2)}},x) \right) \\
     &\quad + f_{\lfloor (n+1)/2 \rfloor}(w) x^{ \lfloor (n+3)/2 \rfloor}  + f_{\lfloor (n+1)/2 \rfloor}(w) x^{ \lfloor (n+5)/2 \rfloor} \\
     &= 2x F(w,x) + f_{\lfloor (n+1)/2 \rfloor}(w) x^{ \lfloor (n+3)/2 \rfloor}(1+x). \qedhere
\end{align*} 
\end{proof}

The recursions in Lemma~\ref{lem:cd-component-recursions} also show by induction that $F(w,x)$ is a symmetric polynomial with centre of symmetry $(n+1)/2$ for every $cd$-monomial $w$ of degree $n$.
For a polynomial $p(x)$ with centre of symmetry $d/2$, its \emph{$\gamma$-polynomial} is the unique polynomial $\gamma_p(t)$ satisfying
$$
 p(x)=(1+x)^d\gamma_p\!\left(\frac{x}{(1+x)^2}\right).
$$
We say that $p$ is $\gamma$-positive if $\gamma_p(t)$ has nonnegative
coefficients.

The lemma may also be used to explicitly derive a closed formula for the right augmented Chow polynomial. Let us first introduce some notation.

Let $C_j=\frac{1}{j+1}\binom{2j}{j}$ be the $j$-th Catalan number. For any natural number $k$, we define 
    \[
        T_k(x) :=  \begin{cases}
        (-1)^{(k-1)/2} C_{(k-1)/2} x^{(k-1)/2} & \text{if $k$ is odd}\\
        0 & \text{if $k$ is even},
        \end{cases}
    \]
and also
\begin{align*}
    S_k(x) &:=  (1+x)^{k+1}
      +x\sum_{i=0}^{k-1}(1+x)^iT_{k-i}(x), \\
    s_k(x) &:= 1 + x\sum_{i=0}^{k-1}T_{k-i}(x).
\end{align*}
Note that the family of polynomials $\{T_k(x)\}_{k \in \mathbb{N}}$ may be found in the formula for the toric $g$-vector of an Eulerian poset found by Bayer and Ehrenborg \cite{bayer-ehrenborg}. Propositions~7.12 and 7.13 of the same article imply that for any $cd$-monomials $u$ and $v$ and any natural number $k$ we have
\begin{align} 
f(u\mathbf{d}v,x) &= f(u,x) f(\mathbf{d}v,x), \label{eq:f(udv)_factorizes}\\
f_{\lfloor (k+1)/2 \rfloor}(\mathbf{c}^k)x^{\lfloor(k+1)/2 \rfloor} &= T_{k+1}(x). \label{eq:f_i(c^k,x)-as-T_k+1(x)}
\end{align}
We now prove the following proposition about $F(w,x)$ and its $\gamma$-polynomial $\gamma_F(w,t)$.
\begin{proposition}\label{prop:gamma-component-formula}
If $w=\mathbf{c}^{k_0}\mathbf{d}\mathbf{c}^{k_1}\mathbf{d}\cdots
\mathbf{d}\mathbf{c}^{k_r}$ is a $cd$-monomial of degree $n$, then
\begin{align*}
    F(w,x) &= x^r \sum_{\ell = 0}^r 2^{\ell} (1+x)^{h_\ell} S_{k_{\ell}}(x) \prod_{j = \ell+1}^{r} T_{k_j+1}(x), \\
    \gamma_F(w,t) &= t^r \sum_{\ell = 0}^r 2^{\ell} s_{k_{\ell}}(t) \prod_{j = \ell+1}^{r} T_{k_j+1}(t),
\end{align*}
where $h_{\ell} := \sum_{i=0}^{\ell-1}k_i$.
\end{proposition}
\begin{proof}
To prove the statement, we only need to show the second identity, as the first will be implied by the definition of the $\gamma$-polynomial. Observe that Lemma \ref{lem:cd-component-recursions} can be reformulated for $\gamma_F$ as follows
\begin{align}
    \gamma_F(\mathbf{c}v,t) &= \gamma_F(v,t) + f_{\lfloor (n+1)/2 \rfloor}(v) t^{\lfloor (n+3)/2 \rfloor} \label{eq:recursion_F(cv)},\\
    \gamma_F(\mathbf{d}v,t) &=2t\gamma_F(v,t) + f_{\lfloor (n+1)/2 \rfloor}(v) t^{\lfloor (n+3)/2 \rfloor}\label{eq:recursion_F(dv)},
\end{align}
for any $cd$-monomial $v$ of degree $n$. In particular, given $w = \mathbf{c}^k$ for some $k \geq 0$, we may repeatedly apply \eqref{eq:recursion_F(cv)} to obtain
\begin{align*}
 \gamma_F(\mathbf{c}^k,t)
 &=1+\sum_{j=0}^{k-1}
 f_{\lfloor(j+1)/2\rfloor}(\mathbf{c}^j)t^{\lfloor(j+3)/2\rfloor}\\
 &=1+t\sum_{j=0}^{k-1}T_{j+1}(t)=s_k(t),
\end{align*}
where we used \eqref{eq:f_i(c^k,x)-as-T_k+1(x)} in the second equality.

Similarly, let $w$ be a $cd$-monomial of degree $n$ of the form $w = \mathbf{c}^k\mathbf{d}v$ for some $k \geq 0$ and some $cd$-monomial $v$. If we apply the recursive identity \eqref{eq:recursion_F(cv)} repeatedly and then \eqref{eq:f(udv)_factorizes}, \eqref{eq:f_i(c^k,x)-as-T_k+1(x)}, and \eqref{eq:recursion_F(dv)}, we get
\begin{align*}
    \gamma_F(\mathbf{c}^k\mathbf{d}v,t) &= \gamma_F(\mathbf{d}v,t) +f(\mathbf{d}v,t)\sum_{j=0}^{k-1}f_{\lfloor(j+1)/2\rfloor}(\mathbf{c}^j)t^{\lfloor(j+3)/2\rfloor}\\
    &= \gamma_F(\mathbf{d}v,t) +tf(\mathbf{d}v,t)\sum_{j=0}^{k-1}T_{j+1}(t)\\
    &=2t\gamma_F(v,t)+f(\mathbf{d}v,t) +tf(\mathbf{d}v,t)\sum_{j=0}^{k-1}T_{j+1}(t)\\
    &=2t\gamma_F(v,t)+s_k(t)f(\mathbf{d}v,t).
\end{align*}
By induction on $r$ and using \eqref{eq:f(dv)}, this proves that for $w=\mathbf{c}^{k_0}\mathbf{d}\mathbf{c}^{k_1}\mathbf{d}\cdots
\mathbf{d}\mathbf{c}^{k_r}$ we have
\[
\gamma_F(w,t) = t^r \sum_{\ell = 0}^r 2^{\ell} s_{k_{\ell}}(t) \prod_{j = \ell+1}^{r} T_{k_j+1}(t). \qedhere
\]
\end{proof}
\subsection{Unimodality for Gorenstein* posets}

In this subsection we prove that the right augmented Chow polynomial of a Gorenstein* poset $P$ is always unimodal.

Write $f(w,x)=\sum_{i\geq 0} f_i(w)x^i$ and $F(w,x)=\sum_{j\geq 0}F_j(w)x^j$.

\begin{proposition}\label{prop:cd-components-unimodal}
Let $w$ be a ${cd}$-monomial of degree $n$. For every $0\leq i\leq\lfloor(n+2)/2\rfloor$, the polynomial $F(w,x)$ satisfies
\begin{equation}\label{eq:strong-component-inequality}
  F_i(w)-F_{i-1}(w) \geq |f_i(w)|,
\end{equation}
where $F_{-1}(w)=0$. In particular, $F(w,x)$ is nonnegative and unimodal.
\end{proposition}

\begin{proof}
For every $cd$-monomial $w$, set $\nabla_i(w)=F_i(w)-F_{i-1}(w)$,
with $F_j(w)=f_j(w)=0$ for $j<0$. We prove the statement by induction on $n$.

If $n=0$, the statement is trivial since $F(1,x)=1+x$ and $f(1,x)=1$. 

If $n >0$, consider the case where $w= \mathbf{c}v$ for some $cd$-monomial $v$. Then, for $i < \lfloor (n+2)/2 \rfloor$, 
\begin{align*}
    \nabla_i(\mathbf{c}v) &= F_i(\mathbf{c}v)-F_{i-1}(\mathbf{c}v) \\
    &=  \nabla_i(v) + \nabla_{i-1}(v),
\end{align*}
by Lemma~\ref{lem:cd-component-recursions}. Using the induction hypothesis, we get
\begin{align*}
    \nabla_i(\mathbf{c}v)&\geq |f_i(v)|+ |f_{i-1}(v)| \\
    &\geq |f_i(v)-f_{i-1}(v)|\\
    &= |f_i(\mathbf{c}v)|,
\end{align*}
where the last equality follows from \eqref{eq:f(cv)}. If $n$ is even and $i=n/2+1$, symmetry of $F(\mathbf{c}v,x)$
gives $\nabla_i(\mathbf{c}v)=0$. The degree bound on
$f(\mathbf{c}v,x)$ gives $f_i(\mathbf{c}v)=0$, so the inequality
holds. If $n$ is odd and $i=(n+1)/2$, we have
\begin{align*}
    \nabla_i(\mathbf{c}v) &= F_i(\mathbf{c}v)-F_{i-1}(\mathbf{c}v) \\
     &= \nabla_{i}(v)+ \nabla_{i-1}(v) + f_{i-1}(v)\\
    &\geq \nabla_{i}(v) \geq |f_i(v)|= |f_i(\mathbf{c}v)|.
\end{align*}
Consider now the case where $w$ is a $cd$-monomial of the form $w = \mathbf{d}v$ for some $cd$-monomial $v$. Notice that for every $i < \lfloor (n+1)/2\rfloor$ we have
\begin{align*}
     \nabla_i(\mathbf{d}v) &= F_i(\mathbf{d}v)-F_{i-1}(\mathbf{d}v)\\
    &=2\nabla_{i-1}(v) \geq 0= |f_i(\mathbf{d}v)|.
\end{align*}
Otherwise, if $i= \lfloor (n+1)/2 \rfloor$
\begin{align*}
    \nabla_i(\mathbf{d}v) &= F_i(\mathbf{d}v)-F_{i-1}(\mathbf{d}v)\\
    &=2\nabla_{i-1}(v) + f_{i-1}(v)\\
    &\geq |f_{i-1}(v)| = |f_i(\mathbf{d}v)|.
\end{align*}
If $n$ is odd, then $\lfloor (n+1)/2 \rfloor = \lfloor (n+2)/2 \rfloor$; otherwise, if $n$ is even and $i=\lfloor (n+2)/2 \rfloor$, we get
\begin{align*}
    \nabla_i(\mathbf{d}v) &= F_i(\mathbf{d}v)-F_{i-1}(\mathbf{d}v)= 0  =|f_i(\mathbf{d}v)|. \qedhere
\end{align*}
\end{proof}

\begin{theorem}\label{thm:gorenstein-augmented-unimodal}
Let $P$ be an Eulerian poset whose ${cd}$-index has nonnegative coefficients.  Then its right and left augmented Chow polynomials are nonnegative and unimodal.  In particular, this holds for every Gorenstein* poset.
\end{theorem}
\begin{proof}
The general statement is an immediate consequence of Proposition~\ref{prop:cd-components-unimodal}. The case of Gorenstein* posets follows from the fact that the $cd$-index of a Gorenstein* poset has nonnegative coefficients \cite{karu}. 
\end{proof}
We may also notice that being Gorenstein* does not imply the $\gamma$-positivity of $F_P(x)$. We show a counterexample below.

\begin{example}\label{ex:augmented-not-gamma-positive}
Let $P$ be the poset shown in Figure \ref{fig:butterfly_rank_4}. 
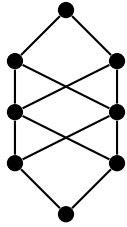
\begin{figure}[ht!] 
    \centering
    \begin{tikzpicture}[baseline=(current bounding box.center),
      scale=0.45,auto=center,
      every node/.style={circle,scale=0.8,fill=black,inner sep=2.7pt}]
      \tikzstyle{edges}=[thick];

      % RANK 4: Maximal element
      \node (n4) at (0,6) {};

      % RANK 3
      \node (n3_a) at (-1.5,4.5) {};
      \node (n3_b) at (1.5,4.5) {};

      % RANK 2
      \node (n2_a) at (-1.5,3) {};
      \node (n2_b) at (1.5,3) {};

      % RANK 1
      \node (n1_a) at (-1.5,1.5) {};
      \node (n1_b) at (1.5,1.5) {};

      % RANK 0: Minimal element
      \node (n0) at (0,0) {};

      % --- EDGES ---

      % R0 to R1
      \draw[edges] (n0) -- (n1_a);
      \draw[edges] (n0) -- (n1_b);

      % R1 to R2
      \draw[edges] (n1_a) -- (n2_a);
      \draw[edges] (n1_a) -- (n2_b);
      \draw[edges] (n1_b) -- (n2_a);
      \draw[edges] (n1_b) -- (n2_b);

      % R2 to R3
      \draw[edges] (n2_a) -- (n3_a);
      \draw[edges] (n2_a) -- (n3_b);
      \draw[edges] (n2_b) -- (n3_a);
      \draw[edges] (n2_b) -- (n3_b);

      % R3 to R4
      \draw[edges] (n3_a) -- (n4);
      \draw[edges] (n3_b) -- (n4);

    \end{tikzpicture}
    \caption{Hasse diagram of the butterfly poset of rank $4$.}
    \label{fig:butterfly_rank_4}
\end{figure}
It is Gorenstein* and its $cd$-index is $\mathbf{c}^3$.
Hence, its right augmented Chow polynomial is
\[
 F_{P}(x)=1+5x+7x^2+5x^3+x^4 =(1+x)^4+x(1+x)^2-x^2.
\]
We observe that the $\gamma$-polynomial $\gamma_{F_{P}}(x)=1+x-x^2$ has a negative coefficient. Since $P$ is self-dual, the same observation holds for the left augmented Chow polynomial.
\end{example}

\subsection{\texorpdfstring{$\mathbf{\gamma}$}{Gamma}-positivity for polytopes}
\label{subsec:polytopal-cd-augmented}

In this subsection we prove that the right augmented Chow polynomial of the face poset of a polytope is always \textit{$\gamma$-positive}. 

Let $\langle \cdot,\cdot\rangle: \Z\langle \mathbf{c},\mathbf{d} \rangle \times \Z \langle \mathbf{c},\mathbf{d} \rangle \to \Z $ denote the bilinear form such that $\langle u,v\rangle=\delta_{u,v}$ for every $cd$-monomial $u$ and $v$.  For every linear functional $\psi: \Z \langle \mathbf{c},\mathbf{d} \rangle_n \to \Z$ on the homogeneous $cd$-polynomials of degree $n$, there exists a homogeneous $cd$-polynomial $u$ of degree $n$ such that $\psi(v)=\langle u,v\rangle$ for every $v \in \Z\langle \mathbf{c},\mathbf{d} \rangle_n$. In this section, we will identify the linear functionals $\psi$ with their corresponding $cd$-polynomial $u$.
In particular, for every $0\leq i\leq\lfloor n/2\rfloor$, let $g_i^{n}$ be the linear functional such that
\begin{equation}\label{eq:toric-g-functional}
 \langle g_i^{n},\Phi_{ P}\rangle=g_i( P),
\end{equation}
for every  Eulerian poset $P$ of rank $n+1$, where $(g_0( P),\dots,g_{\lfloor n/2 \rfloor}(P))$ is the toric $g$-vector of $P$; we also set $g_0^{0}=1$. We prove two facts about this functional. 
\begin{proposition}\label{prop:polytopal-cd-tools}
The following statements hold.
\begin{enumerate}[\normalfont(i)]
 \item For every every $cd$-monomial $w=\mathbf{c}^{k_0}\mathbf{d}\mathbf{c}^{k_1}\mathbf{d}\cdots \mathbf{d}\mathbf{c}^{k_r}$ of degree $2n$, we have
 \begin{equation} \label{eq:formula_top_functional_toric_g}
 [w]g^{2n}_n =\begin{cases}
        (-1)^{k_0/2+\cdots+k_r/2}
        \prod_{i=0}^rC_{k_i/2},
        &k_0,\ldots,k_r\text{ are all even},\\[4pt]
        0,&\text{otherwise}.
    \end{cases}
 \end{equation}
 \item   Let $i$ be a natural number and $u$ be a $cd$-monomial. We have
 \begin{equation} \label{eq:top_functional_left_mult_is_positive}
  \langle ug_i^{2i},\Phi_{P}\rangle \geq 0
 \end{equation}
 for every $P$ which is the face poset of a polytope.
\end{enumerate}
\end{proposition}

\begin{proof}
The statement of (i) follows directly from Theorem~4.2 of \cite{bayer-ehrenborg}. 
In order to prove (ii), we start by noticing that the inequality \eqref{eq:top_functional_left_mult_is_positive} is equivalent to $\langle (g_i^{2i})^{\ast}u^{\ast},\Phi_{P^{\ast}}\rangle \geq 0$. Hence, we may prove that
\[
\langle (g_i^{2i})^{\ast}u,\Phi_{P}\rangle \geq 0
\]
for every $cd$-monomial $u$ and every face poset $P$ of a polytope. Since $(g_i^{2i})^{\ast}$ is equal to $g_i^{2i}$ \cite[Theorem~5.1]{bayer-ehrenborg}, applying Theorem~4.4 of \cite{ehrenborg-lifting} proves the statement for $i \geq 2$. If $i=0$, then $g^0_0=1$, so the statement follows from the nonnegativity of the $cd$-index of the face poset of a polytope \cite{karu}. If $i=1$, then one may calculate that $g^2_1=\mathbf{d}-\mathbf{c}^2$. Using Theorem~3.7 of \cite{ehrenborg-lifting}, we get that $\langle u\mathbf{d},\Phi_P\rangle  \geq\langle u\mathbf{c}^2,\Phi_P\rangle$, which proves the remaining case. 
\end{proof}
We now study the $\gamma$-polynomial of the right augmented Chow polynomial. Denote by $\mathcal{W}_n$ the set of ${cd}$-monomials of degree $n$. For any natural number $n$ and any $0\leq q\leq\lfloor(n+1)/2\rfloor$, we define $L_{n,q}:=\sum_{w\in\mathcal{W}_n}[t^q]\gamma_F(w,t)\,w$. Thus,
\[
\gamma_{F_P}(t) = \sum_{q= 0}^{\lfloor(n+1)/2\rfloor} L_{n,q},\Phi_P\rangle \, t^q,
\]
where $P$ is the face poset of an $n$-dimensional polytope. 

For any natural number $\ell$ and $a=(a_0,\ldots,a_{\ell-1})\in\mathbb Z_{\geq0}^\ell$, set also $s(a):=a_0+\cdots+a_{\ell-1}$ and $u(a):=\mathbf{c}^{a_0}\mathbf{d}\mathbf{c}^{a_1}\mathbf{d}\cdots \mathbf{c}^{a_{\ell-1}}\mathbf{d}$, with $s(\emptyset)=0$ and $u(\emptyset)=1$. We now give a formula for the $cd$-polynomial corresponding to $L_{n,q}$.

\begin{proposition}\label{prop:Lnq-positive-decomposition}
Let $1\leq q\leq\lfloor(n+1)/2\rfloor$ and set $p:=n-2q+2$.  Then
\begin{align} \label{eq:Lnq-positive-decomposition}
 L_{n,q}
 ={}&\sum_{\ell=0}^{q-1} 2^\ell \left(
 \sum_{\substack{a\in\mathbb Z_{\geq0}^{\ell}\\ p-s(a) \geq 1}} 
 u(a)\,\mathbf{c}^{p-s(a)} g_{q-\ell-1}^{2(q-\ell-1)} 
 +
 \sum_{\substack{a\in\mathbb Z_{\geq0}^{\ell}\\ p-s(a) \geq 2}} 
 u(a)\,\mathbf{c}^{p-s(a)-2}\mathbf{d} g_{q-\ell-1}^{2(q-\ell-1)} \right)
 \\
 &+2^q
 \sum_{\substack{a_0,\ldots,a_q\geq0\\a_0+\cdots+a_q=n-2q}}
 \mathbf{c}^{a_0}\mathbf{d}\mathbf{c}^{a_1}\mathbf{d}\cdots
 \mathbf{d}\mathbf{c}^{a_q}.
 \notag
\end{align}
\end{proposition}

\begin{proof}
Let $w=\mathbf{c}^{a_0}\mathbf{d}\mathbf{c}^{a_1}\mathbf{d}\cdots
\mathbf{d}\mathbf{c}^{a_r}$.  
By Proposition~\ref{prop:gamma-component-formula}, we have
\begin{align*}
[t^q] \gamma_F(w,t) &= [t^q] \left( t^r \sum_{\ell = 0}^r 2^{\ell} s_{a_{\ell}}(t) \prod_{j = \ell+1}^{r} T_{a_j+1}(t) \right)\\
&=  \sum_{\ell = 0}^r 2^{\ell} [t^{q-r}]\left(  s_{a_{\ell}}(t) \prod_{j = \ell+1}^{r} T_{a_j+1}(t)\right).
\end{align*}
Let $A_{\ell,q}(w):=[t^{q-r}]\left(s_{a_{\ell}}(t) \prod_{j = \ell+1}^{r} T_{a_j+1}(t)\right)$ for every $q\geq 0$ and $\ell=0,\dots,r$. Notice that 
\[
L_{n,q} =  \sum_{\substack{w \in \mathcal{W}_n}}\left(\sum_{\ell = 0}^{r} 2^{\ell} A_{\ell,q}(w) \right)w.
\]
Given a word $w=\mathbf{c}^{a_0}\mathbf{d}\mathbf{c}^{a_1}\mathbf{d}\cdots
\mathbf{d}\mathbf{c}^{a_r}$ and an integer $\ell = 0,\dots,r$, we have that $A_{\ell,q}(w)$ equals zero if  $a_{\ell+1},\dots,a_r$ are not all even numbers. Hence, we may assume that they are even, and if we set $h := q-r-\frac{ a_{\ell+1} + \cdots + a_r}{2}$, we have
\begin{equation*}
    A_{\ell,q}(w) = \begin{cases}
        (-1)^{ a_{\ell+1}/2+\dots+a_r/2 }\prod_{j=\ell+1}^rC_{a_j/2},
        &h=0,\\[3pt]
        (-1)^{a_{\ell+1}/2+\dots+a_r/2+h-1}C_{h-1}
        \prod_{j=\ell+1}^rC_{a_j/2},
        &h\geq1\text{ and }a_\ell\geq2h-1,\\[3pt]
        0,&\text{otherwise},
    \end{cases} 
\end{equation*}
which, by Proposition~\ref{prop:polytopal-cd-tools} (i), is equivalent to
\begin{equation*}
    A_{\ell,q}(w) = \begin{cases}
        [\mathbf{c}^{a_{\ell+1}}\mathbf{d}\mathbf{c}^{a_{\ell+2}}\mathbf{d}\cdots
\mathbf{d}\mathbf{c}^{a_r}]g^{a_{\ell+1}+\dots+a_r+2(r-\ell-1)}_{a_{\ell+1}/2+\dots+a_r/2+r-\ell-1},
        &h=0\text{ and }\ell < r\\[4pt]
        1,
        &h=0\text{ and }\ell = r,\\[4pt]
        [\mathbf{c}^{2h-2}\mathbf{d}\mathbf{c}^{a_{\ell+1}}\mathbf{d}\mathbf{c}^{a_{\ell+2}}\mathbf{d}\cdots
\mathbf{d}\mathbf{c}^{a_r}]g^{2(q-\ell-1)}_{q-\ell-1},
        &h\geq1\text{ and }a_\ell\geq2h-1,\\[4pt]
        0,&\text{otherwise}.
    \end{cases} 
\end{equation*}
We now collect the contributions of these terms by dividing them into the cases above. Suppose  $w=\mathbf{c}^{a_0}\mathbf{d}\mathbf{c}^{a_1}\mathbf{d}\cdots
\mathbf{d}\mathbf{c}^{a_r}$ is a $cd$-monomial with $h \geq 1$ and $a_\ell\geq2h-1$ and $a=(a_0,a_1,\dots,a_{\ell-1})$. Notice that $(p-s(a))+(2h-2)=a_{\ell}$. Then, $w$ can be written as
\[
w=u(a) \mathbf{c}^{p-s(a)} \mathbf{c}^{2h-2}\mathbf{d}\mathbf{c}^{a_{\ell+1}}\mathbf{d}\cdots\mathbf{d}\mathbf{c}^{a_r}.
\]
Thus, these terms contribute
\[
    \sum_{\ell=0}^{q-1}2^\ell
    \sum_{\substack{a\in\mathbb Z_{\geq0}^{\ell}\\p-s(a)\geq1}}
    u(a)\,\mathbf{c}^{p-s(a)} g_{q-\ell-1}^{2(q-\ell-1)}.
\]
Similarly, if $w=\mathbf{c}^{a_0}\mathbf{d}\mathbf{c}^{a_1}\mathbf{d}\cdots
\mathbf{d}\mathbf{c}^{a_r}$ is a $cd$-monomial with $h=0$ and $\ell<r$ and $a=(a_0,a_1,\dots,a_{\ell-1})$, $w$ can be written as
\[
w=u(a) \mathbf{c}^{p-s(a)-2} \mathbf{d}\mathbf{c}^{a_{\ell+1}}\mathbf{d}\cdots\mathbf{d}\mathbf{c}^{a_r},
\]
and these terms contribute
\[
\sum_{\ell=0}^{q-1}2^\ell
    \sum_{\substack{a\in\mathbb Z_{\geq0}^{\ell}\\p-s(a)\geq2}}
    u(a)\,\mathbf{c}^{p-s(a)-2}\mathbf{d}
    g_{q-\ell-1}^{2(q-\ell-1)}.
\]
The last remaining case with a non-zero contribution is that of $cd$-monomials with $h=0$ and $\ell=r$, whose contribution is
\[
2^q\sum_{\substack{a_0,\ldots,a_q\geq0\\
                       a_0+\cdots+a_q=n-2q}}
    \mathbf{c}^{a_0}\mathbf{d}\mathbf{c}^{a_1}\mathbf{d}\cdots
    \mathbf{d}\mathbf{c}^{a_q}.
\]
By adding the three expressions, we get the formula in the statement.
\end{proof}

\begin{theorem}\label{thm:polytopal-augmented-gamma}
Let $P$ be the face poset of a polytope. The right and left augmented Chow polynomials of $P$ are $\gamma$-positive.
\end{theorem}

\begin{proof}
Let $P$ be the face poset of a polytope of dimension $n$. It is sufficient to prove the $\gamma$-positivity of the right augmented polynomial, as the left augmented polynomial of $P$ is the right augmented polynomial of the dual $P^{\ast}$. 

Note that the constant coefficient of $\gamma_{F_P}(x)$ is $\langle L_{n,0},\Phi_P\rangle
 =\langle\mathbf{c}^n,\Phi_P\rangle=1$.
Let $q$ be an integer such that $1\leq q\leq\lfloor(n+1)/2\rfloor$. Each term in the first two sums of \eqref{eq:Lnq-positive-decomposition} is a functional that is nonnegative on $\Phi_P$ by Proposition~\ref{prop:polytopal-cd-tools} (ii). Similarly, the last sum is nonnegative on $\Phi_P$ because the $cd$-index of a polytope has nonnegative coefficients \cite{karu}. Consequently, $\langle L_{n,q},\Phi_P\rangle\geq0$ for every $q$, proving that $F_P(x)$ is  $\gamma$-positive.
\end{proof}

\section{Counterexamples}\label{sec:counterexamples}

\subsection{Failure of nonnegativity for Eulerian posets}
By \cite[Corollary~4.6 and Theorem~4.18]{ferroni-matherne-vecchi}, for the characteristic kernel, Chow and left augmented Chow functions are known to be nonnegative (and, in fact, also unimodal). It is thus natural to wonder if the same property holds to Chow and augmented Chow polynomials in the Eulerian case. For Eulerian Chow polynomials, this was already raised by Ferroni, Matherne, and Vecchi in \cite[Question~5.5]{ferroni-matherne-vecchi}. The answer to both questions is that they may have negative coefficients without further conditions on the poset. 

\begin{theorem}
    There exist Eulerian posets whose Chow polynomial and Right augmented Chow polynomials have negative coefficients.
\end{theorem}

The specific question for Chow polynomials was also raised in MathOverflow in July 2024\footnote{See \url{https://mathoverflow.net/questions/475156/eulerian-posets-and-order-complexes}.} receiving no answers. 

\begin{example}\label{ex:counterexample-eulerian-positivity}
We make use of the constructions and the notation of Bayer and Hetyei in \cite{bayer_flag_1999}. 
Let $P$ be the horizontal double of $P(n, \mathcal{I},N)$ for $n=8$, $\mathcal{I}=\{[1,2],[3,6],[5,8]\}$ and $N=19$. Since $\mathcal{I}$ is an even interval system on $[1,n]$, $P$ is an Eulerian poset. 
One may calculate that its Eulerian Chow polynomial is
\[
1 + 1664 x + 128620x^2 +810560x^3 - 125786x^4 +810560x^5 +128620x^6 +1664x^7 +x^8
\]
and its right augmented Chow polynomial is
\[
1+1702x+136410x^2 +1182162x^3 - 1108x^4 - 1108x^5 +1182162x^6 +136410x^7 +1702x^8 +x^9.
\]
Both of the above polynomials have negative coefficients.
\end{example}

\subsection{Failure of log-concavity for face posets of polytopes}
\label{subsec:polytopal-counterexamples}
Using a construction found by ChatGPT-6.2 Astra, we give a negative answer to the question of Brenti and Welker \cite[Question~1]{brenti-welker} on the real-rootedness of the $h$-polynomial of the barycentric subdivision of a polytope boundary.
For the convex $140$-polytope $P$ in Theorem~\ref{thm:polytopal-chow-not-log-concave}, the $f$- and $h$-polynomials of $\operatorname{sd}(\partial P)$ both fail log-concavity in degree $2$, as does the right augmented Eulerian Chow polynomial of its face poset.

\subsubsection{The construction}
For $n\geq d=2r\geq 2$, let $C_d^n$ be the neighborly cubical $d$-polytope whose $(r-1)$-skeleton is isomorphic to that of the $n$-cube, as constructed by Joswig and Ziegler~\cite[Theorem~16]{joswig-ziegler-cubical}. Denote by $\square^m$ the $m$-cube. 

Consider the polytope $P_0 = C^n_d$, then choose a facet $F_0 \cong \square^{d-1}$ and a ridge $R_0 \cong \square^{d-2}$ contained in $F_0$. Then, for each $j=1,\dots,d-2$, we iterate the following process:
\begin{itemize}
    \item construct the polytope $P_j$ by applying a stellar subdivision at the facet $F_{j-1}$, thereby adding a new apex vertex $v_{j}$ to $P_{j-1}$;
    \item take a ridge $R_{j-1} \subseteq F_{j-1}$ of $P_{j-1}$ that contains all the apex vertices of $F_{j-1}$;
    \item take the facet $F_j \subseteq P_j$ defined as the join $R_{j-1}\star v_j$.
\end{itemize}
Note that at each step $j \geq 1$, the ridge $R_j$ is combinatorially equivalent to $\operatorname{Pyr}^{j-1}(\square^{d-1-j})$, while $F_j$ is equivalent to $\operatorname{Pyr}^{j}(\square^{d-1-j})$.
In particular, we conclude the iteration by obtaining a polytope $P_{d-2}$ with a facet $F_{d-2}$ that is equivalent to a simplex. Figure~\ref{figure:successive_subdivisions} shows the choice of facets and ridges, and the resulting sequence of $P_i$ for $n=d=4$.

\begin{figure}[ht!]
    \centering
    \tdplotsetmaincoords{70}{108}
    \begin{tikzpicture}[
        vertex/.style={circle,fill=black,inner sep=2pt},
        schvertex/.style={circle,fill=black!55,inner sep=1.15pt},
        selectedvertex/.style={circle,fill=black,inner sep=1.3pt},
        contextedge/.style={draw=black!40,line width=0.45pt},
        selectededge/.style={draw=MidnightBlue,line width=0.9pt},
        every node/.style={font=\small}
    ]

    % ==================================================
    % Lower row: the selected facets and ridges.
    % ==================================================

    \begin{scope}[yshift=-4.45cm]
    \begin{scope}[tdplot_main_coords,scale=1.1]
        \coordinate (B1) at (1,1,0);
        \coordinate (B2) at (-1,1,0);
        \coordinate (B3) at (-1,-1,0);
        \coordinate (B4) at (1,-1,0);
        \coordinate (T1) at (1,1,2);
        \coordinate (T2) at (-1,1,2);
        \coordinate (T3) at (-1,-1,2);
        \coordinate (T4) at (1,-1,2);

        \draw[dotted] (B2) -- (B3) -- (B4);
        \draw[dotted] (B3) -- (T3);
        \fill[red!50,fill opacity=0.55]
            (B4) -- (B1) -- (T1) -- (T4) -- cycle;
        \draw[thick] (B4) -- (B1) -- (B2);
        \draw[thick] (T1) -- (T2) -- (T3) -- (T4) -- cycle;
        \draw[thick] (B1) -- (T1) (B2) -- (T2) (B4) -- (T4);
        \foreach \v in {B1,B2,B3,B4,T1,T2,T3,T4}
            \node[vertex] at (\v) {};
        \coordinate (R0) at (barycentric cs:B4=1,B1=1,T1=1,T4=1);
        \node[BrickRed] at (R0) {\contour{white}{$R_0$}};
    \end{scope}
    \node[below] at (0,-0.7) {$F_0$};

    \draw[-{Stealth[inset=1pt,scale length=1.5]},thick]
        (1.6,0.85) -- node[above] {%subdivide
        }
        node[below] {%at $F_0$
        } (3.1,0.85);

    \begin{scope}[xshift=4.7cm]
        \begin{scope}[tdplot_main_coords,scale=1.1]
            \coordinate (P1) at (1,1,0);
            \coordinate (P2) at (-1,1,0);
            \coordinate (P3) at (-1,-1,0);
            \coordinate (P4) at (1,-1,0);
            \coordinate (V1) at (0,0,2);

            \draw[dotted] (P2) -- (P3) -- (P4);
            \draw[dotted] (P3) -- (V1);
            \fill[red!50,fill opacity=0.55]
                (P4) -- (P1) -- (V1) -- cycle;
            \draw[thick] (P4) -- (P1) -- (P2);
            \draw[thick] (P1) -- (V1) (P2) -- (V1) (P4) -- (V1);
            \foreach \v in {P1,P2,P3,P4}
                \node[vertex] at (\v) {};
            \node[vertex,label=above:{$v_1$}] at (V1) {};
            \coordinate (R1) at (barycentric cs:P4=1,P1=1,V1=1);
            \node[BrickRed] at (R1) {\contour{white}{$R_1$}};
        \end{scope}
        \node[below] at (0,-0.7) {$F_1$};
    \end{scope}

    \draw[-{Stealth[inset=1pt,scale length=1.5]},thick]
        (6.3,0.85) -- node[above] {%subdivide
        }
        node[below] {%at $F_1$
        } (7.8,0.85);

    \begin{scope}[xshift=9.4cm]
        \begin{scope}[tdplot_main_coords,scale=1.1]
            \coordinate (A) at (1,-1,0);
            \coordinate (B) at (1,1,0);
            \coordinate (U1) at (0,0,2);
            \coordinate (V2) at (0,1.2,1.05);

            \draw[dotted] (A) -- (V2);
            \fill[black!5,fill opacity=0.0] (A) -- (B) -- (U1) -- cycle;
            \draw[thick] (A) -- (B) -- (V2) -- (U1) -- cycle;
            \draw[thick] (B) -- (U1);
            \node[vertex] at (A) {};
            \node[vertex] at (B) {};
            \node[vertex,label=above:{$v_1$}] at (U1) {};
            \node[vertex,label=right:{$v_2$}] at (V2) {};
        \end{scope}
        \node[below] at (0,-0.7) {$F_2 = \Delta_3$};
    \end{scope}
    \end{scope}

    % ==================================================
    % Upper row: Schlegel diagrams in a fixed outer cube.
    % I1,I2,I3,I4 are the bottom vertices of F_0;
    % I5,I6,I7,I8 are their corresponding top vertices.
    % R_0=conv(I4,I1,I5,I8), and R_1=conv(I4,I1,S1).
    % ==================================================

    \foreach \stage/\offset in {0/0cm,1/4.7cm,2/9.4cm}{
        \begin{scope}[xshift=\offset,yshift=0.85cm]
            \begin{scope}[tdplot_main_coords]
                % Nested cubes, with corresponding vertices.
                \foreach \prefix/\a in {O/1.24,I/0.68}{
                    \coordinate (\prefix1) at ( \a, \a,-\a);
                    \coordinate (\prefix2) at (-\a, \a,-\a);
                    \coordinate (\prefix3) at (-\a,-\a,-\a);
                    \coordinate (\prefix4) at ( \a,-\a,-\a);
                    \coordinate (\prefix5) at ( \a, \a, \a);
                    \coordinate (\prefix6) at (-\a, \a, \a);
                    \coordinate (\prefix7) at (-\a,-\a, \a);
                    \coordinate (\prefix8) at ( \a,-\a, \a);
                }

                % Schlegel images of the added vertices.
                % S1 is interior to F_0; S2 is interior to F_1.
                \coordinate (S1) at (0,0,0);
                \coordinate (S2) at (0.54,-0.32,0.24);

                % Selected facet shading (drawn beneath edges).
                \ifnum\stage=0
                    \fill[MidnightBlue!8] (I1)--(I2)--(I6)--(I5)--cycle;
                    \fill[MidnightBlue!8] (I5)--(I6)--(I7)--(I8)--cycle;
                    \fill[MidnightBlue!8] (I4)--(I1)--(I5)--(I8)--cycle;
                \else
                    \ifnum\stage=1
                        \fill[MidnightBlue!12,fill opacity=0.65]
                            (I4)--(I1)--(I5)--(I8)--cycle;
                        \fill[MidnightBlue!12,fill opacity=0.65]
                            (I4)--(I1)--(S1)--cycle;
                    \else
                        \fill[MidnightBlue!16,fill opacity=0.65]
                            (I4)--(I1)--(S1)--cycle;
                        \fill[MidnightBlue!24,fill opacity=0.65]
                            (I4)--(I1)--(S2)--cycle;
                    \fi
                \fi

                % Unchanged edges: 12 outer, 12 inner, 8 radial.
                \foreach \prefix in {O,I}{
                    \draw[contextedge]
                        (\prefix1)--(\prefix2)--(\prefix3)--(\prefix4)--cycle;
                    \draw[contextedge]
                        (\prefix5)--(\prefix6)--(\prefix7)--(\prefix8)--cycle;
                    \foreach \u/\v in {1/5,2/6,3/7,4/8}
                        \draw[contextedge] (\prefix\u)--(\prefix\v);
                }
                \foreach \u in {1,...,8}
                    \draw[contextedge] (O\u)--(I\u);

                % First subdivision: join S1 to all 8 vertices of F_0.
                \ifnum\stage>0
                    \foreach \u in {1,...,8}
                        \draw[contextedge] (S1)--(I\u);
                \fi

                % Second subdivision: join S2 to the 5 vertices of F_1.
                \ifnum\stage>1
                    \foreach \u in {1,4,5,8}
                        \draw[contextedge] (S2)--(I\u);
                    \draw[contextedge] (S2)--(S1);
                \fi

                % Blue edges distinguish the selected facet F_stage.
                \ifnum\stage=0
                    \draw[selectededge] (I1)--(I2)--(I3)--(I4)--cycle;
                    \draw[selectededge] (I5)--(I6)--(I7)--(I8)--cycle;
                    \foreach \u/\v in {1/5,2/6,3/7,4/8}
                        \draw[selectededge] (I\u)--(I\v);
                \else
                    \ifnum\stage=1
                        \draw[selectededge] (I4)--(I1)--(I5)--(I8)--cycle;
                        \foreach \u in {1,4,5,8}
                            \draw[selectededge] (S1)--(I\u);
                    \else
                        \draw[selectededge] (I4)--(I1)--(S1)--cycle;
                        \draw[selectededge] (I4)--(S2)--(I1) (S2)--(S1);
                    \fi
                \fi

                % Draw all vertices, including the new apices.
                \foreach \u in {1,...,8}{
                    \node[schvertex] at (O\u) {};
                    \node[schvertex] at (I\u) {};
                }
                \ifnum\stage=0
                    \foreach \u in {1,...,8}
                        \node[selectedvertex] at (I\u) {};
                \else
                    \ifnum\stage=1
                        \foreach \u in {1,4,5,8}
                            \node[selectedvertex] at (I\u) {};
                    \else
                        \foreach \u in {1,4}
                            \node[selectedvertex] at (I\u) {};
                    \fi
                \fi
                \ifnum\stage>0
                    \node[selectedvertex] at (S1) {};
                \fi
                \ifnum\stage>1
                    \node[selectedvertex] at (S2) {};
                \fi

                % Label the selected facet, rather than its apex vertices.
                \ifnum\stage=0
                    \coordinate (SF) at (0,0,0);
                \else
                    \ifnum\stage=1
                        \coordinate (SF) at (0.50,0,-0.40);
                    \else
                        \coordinate (SF) at
                            (barycentric cs:I4=2,I1=2,S1=1,S2=1);
                    \fi
                \fi
                \node[MidnightBlue] at (SF) {\contour{white}{$F_{\stage}$}};
            \end{scope}

            \ifnum\stage=0
                \node[below] at (0,-1.85) {$P_0=C_4^4$};
            \else
                \node[below] at (0,-1.85) {$P_{\stage}$};
            \fi
        \end{scope}
    }

    \draw[-{Stealth[inset=1pt,scale length=1.5]},thick]
        (1.75,0.85) -- (2.95,0.85);
    \draw[-{Stealth[inset=1pt,scale length=1.5]},thick]
        (6.45,0.85) -- (7.65,0.85);

    \end{tikzpicture}
    \caption{\label{figure:successive_subdivisions}
        Construction of $P_i$ for $n=d=4$ and $i=0,1,2$.
        The upper row shows the Schlegel diagrams of the $4$-polytopes, with the selected facets $F_j$ outlined and labeled in blue. The lower row shows the selected facets, with the ridges $R_j$ shaded red.}
\end{figure}
For $N\geq0$, perform $N$ further stellar subdivisions at simplex facets, beginning with $F_{d-2}$. Each subsequent subdivision is at a facet created by the preceding one. We denote the resulting $d$-polytope by $Q_{d,n,N}$. The formulas below are independent of these choices.

\subsubsection{Counting multichains}
For a polytope $Q$, let $W_Q(k)$ count the multichains with $k$ entries in $\overline{\cF(Q)}$, the poset of nonempty proper faces of $Q$, and let $c_i(Q)$ count the strict chains with $i$ entries in the same poset.
We set $W_Q(0)=c_0(Q)=1$. For $k\geq1$, fix a strict chain of $i$ faces, with $1\leq i\leq k$. A multichain with these distinct faces is obtained by dividing its $k$ entries into $i$ nonempty blocks and filling each block with the corresponding face of the strict chain. Summing over the strict
chains gives
\begin{equation}\label{eq:stacking-weak-to-strict}
 W_Q(k)=\sum_{i=1}^k\binom{k-1}{i-1}c_i(Q).
\end{equation}
In particular, writing $w_i=W_Q(i)$, we have
\[
 c_1(Q)=w_1,\qquad c_2(Q)=w_2-w_1,\qquad c_3(Q)=w_3-2w_2+w_1.
\]

\begin{lemma}\label{lem:stacking-weak-chains}
Let $Q'$ be the stellar subdivision of a polytope $Q$ at a facet $F$. Then
\begin{equation}\label{eq:stacking-weak-increment}
 W_{Q'}(k)-W_Q(k)=kW_F(k)
\end{equation}
for every $k\geq1$.
\end{lemma}

\begin{proof}
We use the face description in Section~\ref{sec:geometric_interpretation} and write $v$ for the new vertex. The multichains lost in the subdivision are
those with $k-i$ final copies of $F$, preceded by an $i$-term
multichain in the nonempty proper faces of $F$, for $0\leq i\leq k-1$.
These are in bijection with the new multichains whose first entry is $\{v\}$. Indeed, send the copies of $F$ to $k-i$ initial copies of $\{v\}$, followed by the pyramids over the $i$ remaining entries, in order.
In every remaining new multichain, the first face containing $v$ and all faces after it strictly contain $\{v\}$. Deleting $v$ from these faces gives a multichain with $k$ entries of nonempty proper faces of $F$.
Conversely, each such multichain gives $k$ new multichains by choosing an entry and adjoining $v$ to that face and every face after it. This proves the identity.
\end{proof}

For the starting polytope $C_d^n$, let $f_p(C_d^n)$ be the number of $p$-dimensional faces. Grouping multichains by their last face gives, for $k\geq1$,
\[
 W_{C_d^n}(k)=\sum_{p=0}^{d-1} f_p(C_d^n)(2k-1)^p.
\]
Indeed, there are $(2k-1)^p$ multichains with $k$ nonempty entries ending at a specified $p$-cube
\cite[proof of Theorem~3.2]{athanasiadis-cubical-barycentric}.
Major's formula for the face numbers \cite[Proposition~3]{major-neighborly-cubical} can be written as
\begin{equation}\label{eq:stacking-cubical-faces}
 f_p(C_d^n)=2^{n-p}\sum_{i=0}^p
       \binom{n-d+u_i}{u_i}\binom{d-1-i}{p-i}, \qquad u_i=\min(i,d-1-i).
\end{equation}
Substituting these face numbers, interchanging the sums, and applying the binomial theorem gives
\begin{equation}\label{eq:stacking-weak-seed}
 W_{C_d^n}(k)
 =2^{n-d+1}\sum_{i=0}^{d-1}\binom{n-d+u_i}{u_i}
         (2k-1)^i(2k+1)^{d-1-i}.
\end{equation}

For the preparatory facet $F_j=\operatorname{Pyr}^j(\square^{d-1-j})$, let $A_j(k)=\zeta_{\cF(F_j)}(k+1)$, the number of multichains with $k$ entries in the full face poset of $F_j$. Since taking a pyramid multiplies the face poset by a two-element chain, we have
\begin{equation}\label{eq:stacking-full-weak-facet}
 A_j(k)=(k+1)^j\sum_{i=0}^k(2i+1)^{d-1-j}.
\end{equation}
Indeed, a multichain in the full face poset of a $q$-cube, with $i$ entries after the initial empty faces, has $(2i+1)^q$ possibilities.
In each coordinate, one chooses either to leave it free throughout or to fix it to one of the two endpoints for a nonempty initial segment.

Excluding both endpoints of the full face poset now gives
\[
 W_{F_j}(k)=A_j(k)-2A_j(k-1)+A_j(k-2),
\]
where $A_j(-1)=0$. For a simplex facet $F$, its full face poset is Boolean of rank $d$.
Hence $W_F(k)=(k+1)^d-2k^d+(k-1)^d$. Applying Lemma~\ref{lem:stacking-weak-chains} at each step yields
\begin{equation}\label{eq:stacking-weak-counts}
\begin{split}
 W_{Q_{d,n,N}}(k)
 &=W_{C_d^n}(k)
   +k\sum_{j=0}^{d-3}\bigl(A_j(k)-2A_j(k-1)+A_j(k-2)\bigr)\\
 &\quad+Nk\bigl((k+1)^d-2k^d+(k-1)^d\bigr).
\end{split}
\end{equation}

\subsubsection{The counterexamples}

Let $Q$ be a convex $d$-polytope. Write $C_Q(x)=\sum_{i=0}^d c_i(Q)x^i$ for the chain polynomial of $\overline{\cF(Q)}$, equivalently the polynomial that enumerates the number of faces of its order complex by the number of vertices. In the notation of Section~\ref{sec:geometric_interpretation}, $c_i(Q)$ for $i\geq1$ is the sum of the flag $f$-numbers $f_Q(I)$ with $|I|=i$. The chain polynomial of $\cF(Q)$ is $(1+x)^2C_Q(x)$, since the two endpoints can be included independently. The proper order complex of $\cF(Q)$ is the
barycentric subdivision of $\partial Q$. Thus, by \cite[Theorem~5.4]{ferroni-matherne-vecchi},
\begin{equation}\label{eq:stacking-chow-from-chain}
 \H_{\cF(Q)}(x)=h(\operatorname{sd}(\partial Q),x) =\sum_{i=0}^d c_i(Q)x^i(1-x)^{d-i}.
\end{equation}

\begin{theorem}\label{thm:polytopal-chow-not-log-concave}
There exists a convex $140$-polytope $P$ for which both $\H_{\cF(P)}(x)$ and $F_{\cF(P)}(x)$ fail log-concavity. The chain polynomials of $\cF(P)$ and its proper part also fail log-concavity.
\end{theorem}

\begin{proof}
Set $Q=Q_{140,1000,2^{1261}}$. We compute the first three nonconstant coefficients of $C_Q(x)$, $(1+x)^2C_Q(x)$, and $\H_{\cF(Q)}(x)$ from \eqref{eq:stacking-weak-counts} at $k=1,2,3$, using \eqref{eq:stacking-weak-to-strict} and \eqref{eq:stacking-chow-from-chain}.
For $G_{\cF(Q)}(x)$ we use $G=g^{\rev}\cdot\H$ and the toric $g$-recursion \cite[Section~3.16]{stanley-ec1}, grouping faces by combinatorial type.
In each of the four cases, the resulting coefficients $a_1,a_2,a_3$ satisfy the exact bounds
\[
 \frac{10105013}{10^7} < \frac{a_1a_3}{a_2^2} < \frac{10105014}{10^7}.
\]
Thus, log-concavity fails in degree $2$ in each case.

For the polar dual $Q^*$ we have $F_{\cF(Q^*)}(x)=G_{\cF(Q)}(x)$.
Chain reversal preserves the chain polynomials and hence $\H_{\cF(Q^*)}(x)=\H_{\cF(Q)}(x)$. Thus $P=Q^*$ has the required properties.
\end{proof}

By Newton's inequalities, none of these polynomials is real-rooted.
Theorem~\ref{thm:polytopal-chow-not-log-concave} therefore gives a negative answer to the question of Brenti and Welker \cite[Question~1]{brenti-welker} on barycentric subdivisions of polytopal boundaries, and to the question of Athanasiadis and Kalampogia--Evangelinou \cite[Question~5.2]{athanasiadis-kalampogia} for Gorenstein* posets.

The example also disproves the real-rootedness suggested by the computations of Ferroni, Matherne, and Vecchi \cite[Section~5.3]{ferroni-matherne-vecchi} for augmented Chow polynomials of face posets of polytopes. 

\begin{remark}\label{rem:small-gorenstein-chow}
We included this example in a previous version. While this is now improved by the preceding family of examples, we keep it here because the construction deserves to be more well-known. For Gorenstein* posets, failure of log-concavity of the Eulerian Chow
polynomial already occurs in rank $9$. Let $Q_0$ be the join of four
labelled copies of the face poset of a $20$-gon, so its proper part
is the ordinal sum of their proper parts. Its proper order complex is
the join of four $40$-cycles. In copies $i=2,3$ fix
an incident vertex--edge pair $v_i\lessdot e_i$. Fix also an edge
$e_1$ in the first copy and a vertex $u_4$ in the fourth. Starting from $Q_0$, apply the unzipping operation of Murai and
Nevo~\cite[Definition~2.1]{murai-nevo} successively to the covers
\[
 e_3\lessdot u_4,\qquad v_3\lessdot e_3',\qquad
 e_2\lessdot v_3',\qquad v_2\lessdot e_2',
\]
where each primed element denotes the new copy of the lower element created at the preceding
step. Then perform $215$ more
unzippings, beginning with $e_1\lessdot v_2'$ and subsequently using
the cover between the two new elements from the preceding step.
This gives $219$ unzippings in total. The resulting
poset $R$ is Gorenstein*, since $Q_0$ is Gorenstein* and unzipping
preserves this property \cite[Corollary~2.6]{murai-nevo}. The
$cd$-index formula for unzipping \cite[Lemma~2.7]{murai-nevo} gives,
with $A(t)=1+36t$,
\[
 \Phi_R(1,2t)=A(t)^4+4t\bigl(A(t)^2+A(t)\bigr)+430t.
\]
Since $\H_R(x)=\Phi_R(1+x,2x)$, the coefficients in degrees $1$, $2$,
and $3$ are $590$, $11728$, and $233426$, respectively, and
$11728^2-590\cdot233426=-175356<0$.

The proper order complex of $Q_0$ is obtained from the boundary of the
$8$-dimensional cross-polytope by successive edge subdivisions. Each
unzipping corresponds to two edge subdivisions, and edge subdivisions
preserve polytopality \cite{murai-nevo}. Thus $\Delta(\overline R)$ is
the boundary complex of a flag simplicial $8$-polytope with a
non-log-concave $h$-polynomial. This answers a question posed by Ferroni
on MathOverflow.\footnote{ \url{https://mathoverflow.net/questions/496440/flag-spheres-h-polynomials-and-log-concavity}.}
\end{remark}

\begin{table}[H]
\centering
\small
\renewcommand{\arraystretch}{1.25}
\setlength{\tabcolsep}{5pt}
\begin{tabular}{|l|c|c|c|}
\hline
\textbf{property} & \textbf{H} & \textbf{F} & \textbf{G} \\ \hline
coefficientwise nonnegative
 & $\checkmark\ \mathsf{Gor}^\ast$; $\times\ \mathsf{Eul}$
 & $\checkmark\ \mathsf{Gor}^\ast$; $\times\ \mathsf{Eul}$
 & $\checkmark\ \mathsf{Gor}^\ast$; $\times\ \mathsf{Eul}$ \\ \hline
unimodal
 & $\checkmark\ \mathsf{Gor}^\ast$; $\times\ \mathsf{Eul}$
 & $\checkmark\ \mathsf{Gor}^\ast$; $\times\ \mathsf{Eul}$
 & $\checkmark\ \mathsf{Gor}^\ast$; $\times\ \mathsf{Eul}$ \\ \hline
$\gamma$-positive
 & $\checkmark\ \mathsf{Gor}^\ast$; $\times\ \mathsf{Eul}$
 & $\checkmark\ \mathsf{Poly}$; $\times\ \mathsf{Gor}^\ast$
 & $\checkmark\ \mathsf{Poly}$; $\times\ \mathsf{Gor}^\ast$ \\ \hline
log-concave
 & $\times\ \mathsf{Poly}$
 & $\times\ \mathsf{Poly}$
 & $\times\ \mathsf{Poly}$ \\ \hline
real-rooted
 & $\times\ \mathsf{Poly}$
 & $\times\ \mathsf{Poly}$
 & $\times\ \mathsf{Poly}$ \\ \hline
\end{tabular}
\caption{A summary of positivity properties in the Eulerian Chow framework.  Here
$\mathsf{Eul}$, $\mathsf{Gor}^\ast$, and $\mathsf{Poly}$ denote, respectively,
Eulerian posets, Gorenstein* posets, and face posets of convex
polytopes.  A checkmark means that the property holds throughout the
specified class, while a cross means that a counterexample exists in that
class.}
\label{table:eulerianpolyproperties}
\end{table}

\subsection*{AI use statement} Some parts of this paper benefited from the assistance of AI. In particular, as mentioned in the text, the counterexamples described in Section~\ref{sec:counterexamples} were found via ChatGPT-5.6 Sol and ChatGPT-6.2 Astra, with different degrees of substantial mathematical input by the authors. ChatGPT also contributed a proof of Theorem~\ref{thm:polytopal-augmented-gamma} which was later simplified by the authors. The remainder of the mathematical ideas, results, as well as the writing and proof-reading, were a result of human work carried out by the authors.

\subsection*{Acknowledgments}

LF is a member of the GNSAGA group of the Istituto Nazionale di Alta Matematica (INdAM). TF was supported by the Additional Funding Programme for Mathematical Sciences, delivered by EPSRC (EP/V521917/1) and the Heilbronn Institute for Mathematical Research. The authors also thank Roberto Riccardi for useful conversations.

\bibliographystyle{amsalpha}
\bibliography{bibliography}

\end{document}